\documentclass[12pt]{amsart}
\usepackage[a4paper,margin=2.5cm]{geometry}
\usepackage[T1]{fontenc}
\usepackage[english]{babel}
\usepackage{microtype}
\usepackage{amsmath,amssymb,amsthm,amscd,mathtools}
\usepackage{mathrsfs}
\usepackage{enumitem}
\usepackage{xcolor}
\usepackage{xcolor}
\definecolor{citeblue}{RGB}{0,70,140}
\usepackage[
  colorlinks=true,
  linkcolor=citeblue,
  citecolor=citeblue,
  urlcolor=citeblue
]{hyperref}
\usepackage{etoolbox}

\makeatletter
\patchcmd{\@setaddresses}
  {\par\addvspace\bigskipamount\indent}
  {\par\addvspace\bigskipamount\noindent}
  {}{}
\makeatother
\makeatletter
\def\l@subsection{\@tocline{2}{0pt}{3pc}{6pc}{}}
\makeatother
\makeatletter
\renewcommand\section{\@startsection{section}{1}%
  \z@{.7\linespacing\@plus\linespacing}{.5\linespacing}%
  {\normalfont\bfseries\centering}}

\renewcommand\subsection{\@startsection{subsection}{2}%
  \z@{.5\linespacing\@plus.7\linespacing}{-.5em}%
  {\normalfont\bfseries}}
\makeatother

\numberwithin{equation}{section}

\newtheorem{theorem}{Theorem}[section]
\newtheorem{prop}[theorem]{Proposition}
\newtheorem{lemma}[theorem]{Lemma}
\newtheorem{corollary}[theorem]{Corollary}
\newtheorem{definition}[theorem]{Definition}
\newtheorem*{remark}{Remark}
\newtheorem{example}[theorem]{Example}

\DeclareMathOperator{\Hom}{Hom}

\DeclareMathOperator{\Tors}{Tors}
\DeclareMathOperator{\Cob}{Cob}

\title[Equivalence of toral Chern--Simons and Reshetikhin--Turaev theories]{Equivalence of\\ toral Chern--Simons and Reshetikhin--Turaev theories}
\author{Daniel Galviz}
\address{\noindent  YAU MATHEMATICAL SCIENCES CENTER AND DEPARTMENT OF MATHEMATICS, TSINGHUA
UNIVERSITY, BEIJING, CHINA.}
\date{}

\begin{document}

\begin{abstract}
\noindent We prove a natural isomorphism between toral Chern--Simons theory with gauge group $\mathbb T=\mathfrak t/\Lambda\cong U(1)^n$ and the Reshetikhin--Turaev theory associated with the finite quadratic module determined by an even, integral, nondegenerate symmetric bilinear form $K:\Lambda\times\Lambda\to\mathbb Z.$ More precisely, let $G_K=\Lambda^*/K\Lambda$ be the discriminant group of \(K\), equipped with its induced quadratic form \(q_K\), and let \(\mathcal C(G_K,q_K)\) be the corresponding pointed modular category. Using the geometric quantization formulation of toral Chern--Simons theory, we show that the resulting TQFT is naturally isomorphic to the Reshetikhin--Turaev TQFT determined by \(\mathcal C(G_K,q_K)\). The equivalence is established at the level of closed 3--manifold invariants, bordism operators for manifolds with boundary, and the extended \((2+1)\)-dimensional structure, yielding a natural isomorphism of extended TQFTs.
\vspace{0.5cm}
\medskip

\noindent\textbf{Main Theorem.}
Let \(\mathbb T=\mathfrak t/\Lambda\cong U(1)^n\) be a compact torus, let
\(K:\Lambda\times\Lambda\to\mathbb Z\) be an even, integral, nondegenerate symmetric bilinear form, let \((G_K,q_K)\) be the associated finite quadratic module, and let
\(\mathcal C(G_K,q_K)\) denote the corresponding pointed modular category.
Then the Reshetikhin--Turaev theory associated with \(\mathcal C(G_K,q_K)\) is naturally isomorphic, as an extended \((2+1)\)-dimensional TQFT, to the toral Chern--Simons theory with gauge group \(\mathbb T\) and level \(K\).
Equivalently, the two theories define naturally isomorphic symmetric monoidal functors
\[
\Cob^{\mathrm{ext}}_{2+1}\longrightarrow \mathrm{Vect}_{\mathbb C}.
\]
\end{abstract}
\maketitle
{\scriptsize
\setlength{\parskip}{0pt}
\hypersetup{linkcolor=black}
\tableofcontents
}
\newpage
\section{Introduction}
The equivalence between Abelian Chern--Simons theory \cite{Witten:1988} and Reshetikhin--Turaev theory \cite{Reshetikhin:1991} in the \(U(1)\) case was established in~\cite{Galviz1}, where the \(U(1)\) Chern--Simons TQFT at even level \(k\) was shown to be naturally isomorphic to the Reshetikhin--Turaev theory associated with the finite quadratic module \((\mathbb Z_k,q_k)\). It is therefore natural to ask whether the same relationship can be established for arbitrary compact tori. The construction of the toral Chern--Simons extended TQFT in~\cite{Galviz2,Galviz2.5} provides the geometric framework needed to formulate and prove such a higher-rank analogue.

Let $\mathbb T=\mathfrak t/\Lambda \cong U(1)^n$ be a compact torus, and let $K:\Lambda\times\Lambda\to\mathbb Z$ be an even, integral, nondegenerate symmetric bilinear form. This datum has both a geometric and an algebraic meaning. On the geometric side, \(K\) determines the toral Chern--Simons theory studied in~\cite{Galviz2}: for each closed oriented surface \(\Sigma\), one obtains a compact symplectic moduli space of flat \(\mathbb T\)-connections, whose geometric quantization in real polarization recovers the finite group $G_K=\Lambda^*/K\Lambda$ through its Bohr--Sommerfeld data, and hence produces finite-dimensional state spaces together with bordism operators and an extended functorial structure.

On the algebraic side, \(K\) determines the discriminant finite quadratic module $G_K=\Lambda^*/K\Lambda$ equipped with its induced quadratic form \(q_K\), and hence a pointed modular category $\mathcal C(G_K,q_K).$ Applying the Reshetikhin--Turaev construction to this category yields an extended \((2+1)\)-dimensional TQFT \cite{Reshetikhin:1991,Turaev1994}.

The main result of this paper is that these two theories are naturally isomorphic. More precisely, we prove that the toral Chern--Simons extended TQFT associated with \((\mathbb T,K)\) is naturally isomorphic to the  Reshetikhin--Turaev theory associated with the pointed modular category \(\mathcal C(G_K,q_K)\). The comparison is established both for closed \(3\)--manifold invariants and for bordisms with boundary. Thus the discriminant quadratic module \((G_K,q_K)\) provides the precise algebraic datum that matches the geometric toral Chern--Simons theory.

On the Chern--Simons side, one must use the geometric quantization formalism of~\cite{Galviz2}, including the torsion decomposition of the moduli space, the Reidemeister-torsion half-density, and the \(K\)-twisted Maslov correction in the extended theory. On the Reshetikhin--Turaev side, the relevant algebraic object is no longer the cyclic quadratic module \((\mathbb Z_k,q_k)\), but the general discriminant form \((G_K,q_K)\). 

The paper is organized as follows. In Section~\ref{sec:toral-CS}, we recall the toral Chern--Simons extended TQFT of~\cite{Galviz2} and isolate the geometric ingredients needed later: the boundary state spaces, the bordism vectors, the closed partition function, and the extended functoriality. In Section~\ref{section2}, we describe the corresponding Reshetikhin--Turaev theory associated with the finite quadratic module \((G_K,q_K)\). In Section~\ref{section4}, we prove the closed and boundary equivalence theorems and deduce the equivalence of the two extended \((2+1)\)-dimensional TQFTs.

\textbf{Acknowledgements.} I am grateful to Nicolai Reshetikhin for many helpful conversations and suggestions. I would also like to thank Dan Freed, Zixuan Feng, Stavros Garoufalidis, Zhengwei Liu, and Yilong Wang for helpful comments and suggestions on the manuscript.

\newpage

\section{\texorpdfstring{Toral \(U(1)^n\) Chern--Simons Extended TQFT}{Toral U(1)n Chern--Simons Extended TQFT}}
\label{sec:toral-CS}
In this section, we recall the toral Chern–Simons theory in the form needed for the equivalence theorem, following \cite{Galviz2}. Our purpose is not to reproduce the full construction, but to isolate the precise geometric data that enter the comparison with the Reshetikhin–Turaev theory: the boundary Hilbert spaces, the canonical bordism vectors, the closed partition function, and the extended functorial structure. A complementary treatment, based on a rigorous functional-integral construction of toral Chern–Simons theory, was developed in \cite{Galviz2.5}.

Let $\mathbb T=\mathfrak t/\Lambda \cong U(1)^n$
be a compact torus, and let $K:\Lambda\times \Lambda\to\mathbb Z$ be an even, integral,
nondegenerate symmetric bilinear form.  The finite discriminant group
\[
G_K:=\Lambda^*/K\Lambda,
\qquad
\Lambda^*:=\Hom(\Lambda,\mathbb Z),
\]
will appear naturally on the geometric side and will later reappear as the label set for the
pointed modular category on the Reshetikhin--Turaev side.

More concretely, Section \ref{sec:boundary-phase-space} reviews how the symplectic boundary moduli space is quantized in real polarization, producing finite-dimensional state spaces together with canonical BKS comparison operators. Section \ref{toral-tqft} then passes from a closed boundary surface to a compact oriented $3$--manifold $X,$ recalling how the torsion components of the moduli space of flat fields determine translated Lagrangian leaves, canonical sections and half-densities, and hence the normalized bordism vector $Z^{\mathrm{CS}}_{\mathbb T,K}(X)$. We conclude with the cylinder, gluing, and extended-TQFT statements that will be used in Section \ref{section4}.

\subsection{\texorpdfstring{Boundary phase space and real-polarized quantization}{Boundary phase space and real-polarized quantization}}\label{sec:boundary-phase-space}
We begin with the boundary theory, since these are the vector spaces that will later be compared with the Reshetikhin–Turaev state spaces. For a closed oriented surface $\Sigma$, the classical boundary phase space of toral Chern–Simons theory is
\begin{equation}\label{eq:toral-phase-space}
\mathcal M_\Sigma(\mathbb T)
:=
H^1(\Sigma;\mathfrak t)\big/H^1(\Sigma;\Lambda).
\end{equation}
Since \(\mathbb T\) is Abelian, a flat \(\mathbb T\)-connection is determined up to gauge by its
holonomy class, so \(\mathcal M_\Sigma(\mathbb T)\) is precisely the moduli space of flat
boundary fields. The level form \(K\) combines with cup product to define
\begin{equation}\label{eq:toral-symplectic-form}
\omega_{\Sigma,K}([\alpha],[\beta])
=
\int_\Sigma K(\alpha\wedge \beta),
\qquad
[\alpha],[\beta]\in H^1(\Sigma;\mathfrak t).
\end{equation}

\begin{prop}\cite[Proposition 2.1]{Galviz2}\label{prop:toral-phase-space}
For every closed oriented surface \(\Sigma\), the pair $(\mathcal M_\Sigma(\mathbb T),\omega_{\Sigma,K})$ is a compact symplectic torus.
\end{prop}

Thus $\mathcal M_\Sigma(\mathbb T)$ is the toral analogue of the classical boundary phase space in rank--one Abelian Chern–Simons theory. The first step toward quantization is therefore to understand which boundary directions are selected by a bulk $3$--manifold. Now let \(X\) be a compact oriented \(3\)--manifold with boundary \(\partial X=\Sigma\).
The restriction map
\[
r_X:H^1(X;\mathfrak t)\longrightarrow H^1(\Sigma;\mathfrak t)
\]
determines the subspace
\begin{equation}\label{eq:LX-cohomological}
L_X:=\operatorname{Im}(r_X)\subset H^1(\Sigma;\mathfrak t).
\end{equation}
This is the space of boundary classes that extend into the bulk.

\begin{prop}\cite[Proposition~2.2]{Galviz2}\label{prop:LX-Lagrangian}
The subspace \(L_X\subset H^1(\Sigma;\mathfrak t)\) is Lagrangian with respect to
\(\omega_{\Sigma,K}\).
\end{prop}
So the bulk manifold $X$ determines a Lagrangian boundary condition: among all infinitesimal boundary fields, it selects exactly those that arise by restriction from the interior. It is often convenient to express the same datum in homological language. Define
\begin{equation}\label{eq:lambdaX-homological}
\lambda_X:=\ker\!\bigl(H_1(\Sigma;\mathbb R)\to H_1(X;\mathbb R)\bigr).
\end{equation}
Then \(\lambda_X\) is Lagrangian for the surface intersection form. Under Poincar\'e
duality it corresponds to
\begin{equation}\label{eq:LX-real}
L_X^{\mathbb R}:=\operatorname{Im}\!\bigl(H^1(X;\mathbb R)\to H^1(\Sigma;\mathbb R)\bigr),
\qquad
L_X=L_X^{\mathbb R}\otimes \mathfrak t.
\end{equation}

\begin{lemma}\cite[Lemma~2.3]{Galviz2}\label{lem:lambdaX-LX}
Under Poincar\'e duality, the Lagrangian subspace \(\lambda_X\subset H_1(\Sigma;\mathbb R)\)
corresponds to \(L_X^{\mathbb R}\subset H^1(\Sigma;\mathbb R)\), and hence
\[
L_X=L_X^{\mathbb R}\otimes \mathfrak t.
\]
\end{lemma}

This identification is the geometric bridge between the Chern--Simons boundary
polarization and the Lagrangian data used later in the extended bordism formalism. The prequantum geometry is provided by the canonical toral Chern--Simons line bundle.
The toral Chern--Simons construction associates to each boundary field a Hermitian line
obtained by comparing fillings via the \(4\)-dimensional Chern--Weil phase, and these
lines descend from the flat locus to a Hermitian line bundle
\begin{equation}\label{eq:toral-prequantum-line}
\mathcal L_{\Sigma,K}\longrightarrow \mathcal M_\Sigma(\mathbb T)
\end{equation}
with unitary connection \(\nabla_{\Sigma,K}\) satisfying
\begin{equation}\label{eq:toral-prequantum-curvature}
F_{\nabla_{\Sigma,K}}=-2\pi i\,\omega_{\Sigma,K}.
\end{equation}

\begin{prop}\cite[Proposition~2.12]{Galviz2}\label{prop:prequantum-line}
The Hermitian line bundle \(\mathcal L_{\Sigma,K}\to \mathcal M_\Sigma(\mathbb T)\)
with connection \(\nabla_{\Sigma,K}\) is the canonical prequantum line bundle of
\((\mathcal M_\Sigma(\mathbb T),\omega_{\Sigma,K})\).
\end{prop}
Thus the classical boundary phase space comes equipped with exactly the prequantum geometry required for geometric quantization.

To quantize in real polarization, choose a rational Lagrangian
\(L\subset H^1(\Sigma;\mathbb R)\) for the ordinary cup-product symplectic form on
\(H^1(\Sigma;\mathbb R)\). This gives the translation-invariant real polarization
\begin{equation}\label{eq:real-polarization}
\mathcal P_L:=L\otimes \mathfrak t \subset H^1(\Sigma;\mathfrak t).
\end{equation}
The leaves of this polarization are compact subtori. Only some of them satisfy the Bohr–Sommerfeld condition, and these are exactly the leaves that contribute quantum states.
\begin{prop}\cite[Proposition~2.14]{Galviz2}\label{prop:BS-leaves}
Let \(\Sigma_g\) be a connected closed oriented surface of genus \(g\), and let
\(L\subset H^1(\Sigma_g;\mathbb R)\) be a rational Lagrangian.
Then the Bohr--Sommerfeld leaves of the real polarization \(\mathcal P_L\) are finite in
number; in genus one they are parametrized by \(G_K=\Lambda^*/K\Lambda\), and in genus \(g\) they are
parametrized by \(G_K^g=(\Lambda^*/K\Lambda)^g\).
\end{prop}

The resulting Hilbert space is
\begin{equation}\label{eq:toral-hilbert-space}
\mathcal H_{\mathbb T,K}(\Sigma,L)
:=
\bigoplus_{\ell\in \mathcal{BS}(\Sigma,L)}
\Gamma_{\mathrm{flat}}
\!\left(
\ell;\,
\mathcal L_{\Sigma,K}\otimes |\det \mathcal P_L^*|^{1/2}
\right),
\end{equation}
and every Bohr--Sommerfeld leaf contributes a one-dimensional summand.

\begin{theorem}\cite[Theorem~2.17]{Galviz2}\label{thm:toral-quantization-dimension}
For a connected closed oriented surface \(\Sigma_g\) of genus \(g\),
\[
\dim \mathcal H_{\mathbb T,K}(\Sigma_g,L)
=
|G_K|^g
=
|\det K|^g.
\]
\end{theorem}
So the finite discriminant group $G_K$ controls the size of the quantum state space. This is the first place where the lattice datum encoded by $K$ appears in explicitly finite form, and it is precisely this finite group that will later govern the pointed modular category on the Reshetikhin–Turaev side.

The Hilbert space \(\mathcal H_{\mathbb T,K}(\Sigma,L)\) depends a priori on the choice of
rational Lagrangian \(L\), and the comparison between different choices is given by the
Blattner--Kostant--Sternberg (BKS) operators.

\begin{prop}\cite[Proposition~2.19]{Galviz2}\label{lem:BKS-operators}
For any pair of rational Lagrangians \(L_i\subset H^1(\Sigma;\mathbb R)\), the
Blattner--Kostant--Sternberg operator
\[
F_{L_2L_1}:
\mathcal H_{\mathbb T,K}(\Sigma,L_1)\to
\mathcal H_{\mathbb T,K}(\Sigma,L_2)
\]
is a canonical unitary isomorphism.
\end{prop}

Their composition law is projective:
\begin{equation}\label{eq:BKS-composition}
F_{L_3L_2}\circ F_{L_2L_1}
=
e^{\frac{\pi i}{4}\mu_K(L_1,L_2,L_3)}\,F_{L_3L_1}.
\end{equation}
Here
\begin{equation}\label{eq:muK-definition}
\mu_K(L_1,L_2,L_3)
:=
\mu(\mathcal P_{L_1},\mathcal P_{L_2},\mathcal P_{L_3})
=
\sigma(K)\,\mu_\Sigma(L_1,L_2,L_3),
\end{equation}
where \(\mu_\Sigma\) is the usual surface Maslov--Kashiwara index \cite{Manoliu2} and \(\sigma(K)\) is the
signature of \(K\). Thus the failure of strict transitivity is measured not by the ordinary surface Maslov cocycle alone, but by its $K$-twisted toral refinement. This is exactly the anomaly that must be absorbed later by the extended weighting.

\subsection{\texorpdfstring{Torsion sectors, boundary states, and the extended theorem}{Torsion sectors, boundary states, and the extended theorem}}
\label{toral-tqft}
We now pass from quantization of a closed boundary surface to the data attached to a compact oriented $3$--manifold $X$. The key point is that flat toral gauge fields on $X$ decompose by torsion topological type, and each torsion component contributes a translated Lagrangian leaf on the boundary together with a canonical quantum object. The full bordism vector is the normalized sum of these leafwise contributions.

Let \(X\) be a compact oriented \(3\)--manifold. Since \(B\mathbb T\simeq\mathcal K(\Lambda,2)\)\footnote{For \(\mathbb T=\mathfrak t/\Lambda\), the classifying space \(B\mathbb T\) is an Eilenberg--MacLane space \(\mathcal K(\Lambda,2)\).},
the principal \(\mathbb T\)-bundles over \(X\) are classified by \(H^2(X;\Lambda)\). The flat theory is supported only on torsion classes $p \in \Tors H^2(X;\Lambda).$ For each such \(p\), let \(\mathcal M_{X,p}(\mathbb T)\) denote the corresponding moduli space of flat fields.

\begin{prop}\cite[Proposition~3.1(ii)--(iii)]{Galviz2}\label{prop:torsion-moduli} For each \(p \in \Tors H^2(X;\Lambda)\), the moduli space
\(\mathcal M_{X,p}(\mathbb T)\) of flat \(\mathbb T\)-fields with characteristic class \(p\)
is an affine torus modeled on \(H^1(X;\mathfrak t)/H^1(X;\Lambda)\).
\end{prop}

Its boundary image will be denoted by
\[
\Lambda_{X,p}:=\operatorname{Im}\bigl(\mathcal M_{X,p}(\mathbb T)\to \mathcal M_{\partial X}(\mathbb T)\bigr).
\]

\begin{prop}\cite[Proposition~3.3]{Galviz2}\label{prop:boundary-leaves}
For each torsion class \(p\in \Tors H^2(X;\Lambda)\), the image
\(\Lambda_{X,p}\subset \mathcal M_{\partial X}(\mathbb T)\) is a translated Lagrangian
leaf, and all such leaves are parallel with common tangent space \(L_X\).
\end{prop}

So each torsion component lands on a leaf parallel to the boundary polarization determined in Section \ref{sec:boundary-phase-space}. On each torsion component one has two canonical geometric objects. The first is the Chern–Simons section.

\begin{prop}\cite[Theorem~3.4]{Galviz2}\label{prop:CS-section}
For each \(p\in \Tors H^2(X;\Lambda)\), the classical toral Chern--Simons action
determines a covariantly constant section
\[
\sigma_{X,p}:\mathcal M_{X,p}(\mathbb T)\to r_{X,p}^{\,*}\mathcal L_{\partial X,K}.
\]
\end{prop}

The second is the Reidemeister--torsion half-density.
\begin{equation}\label{eq:toral-half-density}
\mu_{X,p}\in
\Gamma\!\left(
\Lambda_{X,p},\,|\det(T^*\Lambda_{X,p})|^{1/2}
\right),
\end{equation}
obtained from the determinant-line formalism and its natural normalization.

\begin{prop}\cite[Definition~3.7 and Proposition~3.8]{Galviz2}\label{prop:torsion-half-density}
For each \(p\in \Tors H^2(X;\Lambda)\), the Reidemeister torsion determines a canonical
translation-invariant half-density
\[
\mu_{X,p}\in
\Gamma\!\left(
\Lambda_{X,p},\,|\det(T^*\Lambda_{X,p})|^{1/2}
\right).
\]
\end{prop}

Their tensor product
\(
\sigma_{X,p}\otimes \mu_{X,p}
\)
is the leafwise quantum contribution of the torsion class \(p\). To assemble the full
boundary state, one sums over all torsion classes with the universal cohomological
normalization exponent
\begin{equation}\label{eq:mX-definition}
m_X
:=
\frac14\Bigl(
\dim H^1(X;\mathbb R)
+\dim H^1(X,\partial X;\mathbb R)
-\dim H^0(X;\mathbb R)
-\dim H^0(X,\partial X;\mathbb R)
\Bigr).
\end{equation}
If \(X\) is closed and connected, this reduces to
\[
m_X=\frac12\bigl(b_1(X)-1\bigr).
\]

\begin{prop}\cite[Proposition~3.12]{Galviz2}\label{thm:toral-boundary-vector}
The canonical toral boundary vector of \(X\) is
\begin{equation}\label{eq:toral-boundary-vector}
Z^{\mathrm{CS}}_{\mathbb T,K}(X)
:=
\frac{|\det K|^{m_X}}{\#\,\Tors H^2(X;\Lambda)}
\sum_{p\in \Tors H^2(X;\Lambda)}
\sigma_{X,p}\otimes \mu_{X,p}
\in
\mathcal H_{\mathbb T,K}(\partial X,L_X^{\mathbb R}).
\end{equation}
\end{prop}

In other words, the bordism state assigned to $X$ is the normalized sum of the quantum contributions from all torsion sectors. The normalization is not a matter of convention: it is forced by the dimension formula, the cylinder normalization, and the gluing law.

When \(\partial X=\varnothing\), the state space is canonically \(\mathbb C\), and
\eqref{eq:toral-boundary-vector} becomes the closed toral partition function
\begin{equation}\label{eq:toral-closed-partition-function}
Z^{\mathrm{CS}}_{\mathbb T,K}(X)
=
\frac{|\det K|^{m_X}}{\#\,\Tors H^2(X;\Lambda)}
\sum_{p\in \Tors H^2(X;\Lambda)}
\int_{\mathcal M_{X,p}(\mathbb T)}
\sigma_{X,p}\,(T_X(\mathfrak t))^{1/2},
\end{equation}
where \((T_X(\mathfrak t))^{1/2}\) is the translation-invariant density induced by the square
root of Reidemeister--Ray--Singer torsion.

\begin{definition}\label{def:toral-closed-partition-function}
For a closed oriented \(3\)--manifold \(M\), we write $Z^{CS,\mathrm{raw}}_{\mathbb T,K}(M)$ for the scalar \eqref{eq:toral-closed-partition-function}, and call it the closed toral Chern--Simons partition function. This is the raw closed invariant, before the insertion of any extended Maslov anomaly weight. By contrast, $Z^{CS}_{\mathbb T,K}(X,L,n)$ denotes the corrected extended toral Chern--Simons invariant assigned to an extended bordism \((X,L,n)\) as in Theorem~\ref{thm:toral-CS-TQFT}. In particular, for a closed extended \(3\)--manifold \((M,n)\), the symbol $Z^{CS}_{\mathbb T,K}(M,n)$ refers to the corresponding corrected extended scalar.
\end{definition}

The bordism state assignment satisfies two structural properties needed later.
First, the cylinder acts by the identity kernel with the following normalization.

\begin{prop}\cite[Proposition~4.1]{Galviz2}\label{prop:toral-cylinder}
For every closed oriented surface \(\Sigma\),
\[
Z^{\mathrm{CS}}_{\mathbb T,K}(\Sigma\times I)
=
|\det K|^{\frac14\dim H^1(\Sigma;\mathbb R)}\,\mathrm{Id}.
\]
\end{prop}

Second, the theory satisfies the following gluing law.

\begin{theorem}\cite[Theorem~4.5]{Galviz2}\label{thm:toral-gluing}
Let \(X^{\mathrm{cut}}\) be obtained from \(X\) by cutting along a closed oriented surface
\(\Sigma\). Then
\[
Z^{\mathrm{CS}}_{\mathbb T,K}(X)
=
\operatorname{Tr}_{\Sigma}\!\left(
Z^{\mathrm{CS}}_{\mathbb T,K}(X^{\mathrm{cut}})
\right),
\]
where the trace is defined by contraction against the cylinder kernel.
\end{theorem}

The matching of the powers of \(|\det K|\) in this gluing formula is encoded by the
cohomological exponent identity underlying the normalization of the theory.

\begin{prop}\cite[Corollary~4.3 and Proposition~4.4]{Galviz2}\label{prop:normalization-exponent}
The normalization exponent \(m_X\) is compatible with cutting and gluing, in the sense
that the powers of \(|\det K|\) appearing in the boundary-state formula, the cylinder
formula, and the gluing formula match exactly under decomposition along closed oriented
surfaces.
\end{prop}

At this point all of the toral data needed later are in place: the state spaces $\mathcal H_{\mathbb T, K}(\Sigma, L)$, the boundary vectors $Z^{\mathrm{CS}}_{\mathbb T,K}(X)$, the closed partition function, and the BKS comparison operators. The final point is that these fit together into an extended TQFT once the projective BKS anomaly is absorbed by the $K$-twisted weighting convention.

\begin{theorem}\cite[Theorem~4.8]{Galviz2}\label{thm:toral-CS-TQFT}
Let \(\mathbb T=\mathfrak t/\Lambda\cong U(1)^n\) be a compact torus, and let \(K:\Lambda\times\Lambda\to\mathbb Z\) be an even, integral, nondegenerate symmetric bilinear form. Adopt the \(K\)-twisted extended bordism convention: the underlying objects are pairs \((\Sigma,L)\) with \(L\subset H^1(\Sigma;\mathbb R)\) a rational Lagrangian, while the boundary-weight correction is governed by the toral Maslov cocycle \(\mu_K\). Here \(n\) denotes the ordinary Walker--Turaev extended weight; the \(K\)-twist is implemented by the factor \(e^{\frac{\pi i}{4}\sigma(K)n}\).

Then the assignments
\[
(\Sigma,L)\longmapsto \mathcal H_{\mathbb T,K}(\Sigma,L),
\qquad
X\longmapsto Z^{\mathrm{CS}}_{\mathbb T,K}(X),
\]
and
\begin{equation}\label{eq:weighted-toral-assignment}
(X,L,n)\longmapsto
Z^{\mathrm{CS}}_{\mathbb T,K}(X,L,n)
:=
e^{\frac{\pi i}{4}\sigma(K)\,n}\,
F_{L\,L_X^{\mathbb R}}
\!\bigl(Z^{\mathrm{CS}}_{\mathbb T,K}(X)\bigr)
\in \mathcal H_{\mathbb T,K}(\partial X,L)
\end{equation}
define a unitary extended \((2+1)\)-dimensional topological quantum field theory. If \(M\) is closed and connected, then for the closed extended manifold \((M,n)\) one has
\[
Z^{\mathrm{CS}}_{T,K}(M,n)
=
e^{\frac{\pi i}{4}\sigma(K)n}
Z^{\mathrm{CS,raw}}_{T,K}(|M|).
\]
In particular, for the zero-weight extended manifold \((M,0)\), this reduces
to the raw closed partition function \eqref{eq:toral-closed-partition-function}.
\end{theorem}

This completes the toral Chern--Simons theory needed for the equivalence with the Reshetikhin--Turaev theory.

\section{General Abelian Reshetikhin–Turaev Theory}~\label{section2}
In this section, we recall the Reshetikhin–Turaev theory on the algebraic side in the precise form needed for the equivalence theorem. Our aim is to extract from the pointed modular category attached to the discriminant finite quadratic module $(G_K,q_K)$ the three pieces of structure that will later be matched with the toral Chern--Simons theory: the closed surgery invariant, its finite-quadratic normalization, and the boundary state spaces and bordism maps of the Maslov anomaly correction for the extended theory. We begin with the general Abelian surgery formalism in order to place the finite case in context, then pass to the finite quadratic module determined by the even lattice $(\Lambda,K)$, and finally record the boundary theory used in Section~\ref{section4}.

\subsection{Reshetikhin–Turaev invariants for general Abelian groups}
\label{sec:3.1}
We begin with the closed surgery formula on the Reshetikhin–Turaev side, first in the broader Abelian surgery formalism and then in the finite quadratic-module case relevant for Section \ref{section4}.

\begin{definition}\label{def:finite-quadratic-module}
Let $G$ be a finite Abelian group. A map $q:G\to  U(1)$ is called \emph{quadratic} if the associated function
\begin{equation}\label{eq:bq-general-revised}
b_q(x,y):=q(x+y)\,q(x)^{-1}q(y)^{-1}
\end{equation}
is a symmetric bicharacter on $G$. The pair $(G,q)$ is called a
\emph{finite quadratic module} if $b_q$ is nondegenerate.
\end{definition}

\begin{prop}\label{prop:pointed-category-general}
Let $(G,q)$ be a finite quadratic module. Then there is an associated pointed ribbon
category $\mathcal C(G,q)$ whose simple objects are indexed by elements of $G$, with
tensor product
\[
R_x\otimes R_y \simeq R_{x+y},\qquad \mathbf \mathrm{id} =R_0,\qquad R_x^*\simeq R_{-x},
\]
and twist
\[
\theta_x=q(x)\,\mathrm{id}_{R_x}.
\]
Its Hopf--link pairing is the bicharacter $b_q$, so the Hopf--link matrix is
\begin{equation}\label{eq:hopf-matrix-general}
S^{\mathrm{Hopf}}_{x,y}=b_q(x,y).
\end{equation}
In particular, $\mathcal C(G,q)$ is modular if and only if $b_q$ is nondegenerate.
\end{prop}

\begin{proof}
This is the standard pointed construction attached to a finite quadratic module. In the
pointed case every simple object has quantum dimension $1$, the braiding and twist are
determined by the bicharacter $b_q$ and quadratic function $q$, and modularity is
equivalent to the nondegeneracy of the Hopf--link matrix
$S^{\mathrm{Hopf}}_{x,y}=b_q(x,y)$. See, in the Abelian setting relevant here,
\cite[Section~5.4]{Galviz1}; for the extended RT formalism used later, compare
\cite[Section~IV.1--IV.3]{Turaev1994}.
\end{proof}

\begin{example}\label{ex:cyclic-general}
Let $k\in 2\mathbb Z_{>0}$ and let $G=\mathbb Z_k$. Define
\begin{equation}\label{eq:qk-general-revised}
q_k(x):=\exp\!\Bigl(\frac{\pi i}{k}x^2\Bigr),\qquad x\in\mathbb Z_k.
\end{equation}
Then $q_k$ is well defined, because for any lift $x\in\mathbb Z$ one has
\[
(x+k)^2-x^2=2kx+k^2\in 2k\mathbb Z
\]
when $k$ is even. The associated bicharacter is
\begin{equation}\label{eq:bk-general-revised}
b_k(x,y)=q_k(x+y)\,q_k(x)^{-1}q_k(y)^{-1}
       =\exp\!\Bigl(\frac{2\pi i}{k}xy\Bigr),
\end{equation}
which is nondegenerate. Hence $(\mathbb Z_k,q_k)$ is a finite quadratic module, and the
corresponding pointed modular category will be denoted by $\mathcal C_k:=\mathcal C(\mathbb Z_k,q_k).$
\end{example}

We now record the closed pointed RT surgery formula in the cyclic case in a form that will
generalize immediately to arbitrary finite quadratic modules.

\begin{prop}\label{prop:cyclic-RT-surgery}
Let $L=L_1\cup\cdots\cup L_m\subset S^3$ be an oriented framed link with symmetric linking
matrix
\[
\mathcal L_L=(\mathcal L_{ij})_{1\le i,j\le m},
\qquad
\mathcal L_{ij}=\operatorname{lk}(L_i,L_j)\ (i\neq j),
\qquad
\mathcal L_{ii}=\operatorname{fr}(L_i),
\]
and let $M_L$ be the closed oriented $3$--manifold obtained by surgery on $L$.
For a coloring $\vec x=(x_1,\dots,x_m)\in(\mathbb Z_k)^m$, the RT evaluation in
$\mathcal C_k$ is
\begin{equation}\label{eq:RT-eval-cyclic-revised}
\langle L_{\vec x}\rangle
=
\prod_{1\le i<j\le m} b_k(x_i,x_j)^{\mathcal L_{ij}}
\prod_{i=1}^m q_k(x_i)^{\mathcal L_{ii}}
=
\exp\!\Bigl(\frac{\pi i}{k}\,{}^t\vec x\,\mathcal L_L\,\vec x\Bigr).
\end{equation}
Consequently, the closed pointed RT invariant is
\begin{equation}\label{eq:finite-RT-cyclic-revised}
Z^{\mathrm{RT,raw}}_{\mathcal C_k}(M_L)
=
k^{-1/2}\,
A_+(k)^{\frac{-m-\sigma(L)}{2}}\,
A_-(k)^{\frac{-m+\sigma(L)}{2}}\,
\sum_{\vec x\in(\mathbb Z_k)^m}
\exp\!\Bigl(\frac{\pi i}{k}\,{}^t\vec x\,\mathcal L_L\,\vec x\Bigr),
\end{equation}
where
\begin{equation}\label{eq:Apm-cyclic-revised}
A_\pm(k):=
\sum_{s\in\mathbb Z_k}\exp\!\Bigl(\pm\frac{\pi i}{k}s^2\Bigr)
\end{equation}
and $\sigma(L)=\operatorname{sign}(\mathcal L_L)$.
\end{prop}

\begin{proof}
The link evaluation is immediate from the pointed ribbon rules:
\[
\langle L_{\vec x}\rangle
=
\prod_{i<j}\exp\!\Bigl(\frac{2\pi i}{k}\mathcal L_{ij}x_ix_j\Bigr)
\prod_i\exp\!\Bigl(\frac{\pi i}{k}\mathcal L_{ii}x_i^2\Bigr),
\]
which combines into \eqref{eq:RT-eval-cyclic-revised} because
\[
{}^t\vec x\,\mathcal L_L\,\vec x
=
\sum_i \mathcal L_{ii}x_i^2+2\sum_{i<j}\mathcal L_{ij}x_ix_j.
\]
The normalized surgery expression
\eqref{eq:finite-RT-cyclic-revised} is the standard pointed RT surgery formula in this
Abelian case. Invariance under Kirby
moves is the usual one: a first Kirby move contributes the one-component Gauss sum
$A_\pm(k)$, which is cancelled by the normalization, while a second Kirby move is handled
by the usual change-of-variables argument in $(\mathbb Z_k)^m$.
\end{proof}

The same argument extends without change to any finite quadratic module.

\begin{prop}\label{prop:finite-RT-general}
Let $(G,q)$ be a finite quadratic module, and let $\mathcal C(G,q)$ be the associated
pointed modular category. For an oriented framed surgery link $L\subset S^3$ with $m$
components and linking matrix $\mathcal L_L$, define
\begin{equation}\label{eq:gauss-general-revised}
A_\pm(G,q):=\sum_{x\in G} q(x)^{\pm1}.
\end{equation}
Then the closed pointed RT invariant of $M_L$ is
{\small\begin{equation}\label{eq:finite-RT-general-revised}
Z^{\mathrm{RT,raw}}_{\mathcal C(G,q)}(M_L)
=
|G|^{-1/2}\,
A_+(G,q)^{\frac{-m-\sigma(L)}{2}}\,
A_-(G,q)^{\frac{-m+\sigma(L)}{2}}\,
\sum_{\vec x\in G^m}
\prod_{1\le i<j\le m} b_q(x_i,x_j)^{\mathcal L_{ij}}
\prod_{i=1}^m q(x_i)^{\mathcal L_{ii}}.
\end{equation}}
This quantity depends only on the oriented homeomorphism type of $M_L$.
\end{prop}

\begin{proof}
Since $\mathcal C(G,q)$ is pointed, every simple object has quantum dimension $1$, so the RT evaluation of a colored surgery link is exactly the product of Hopf--link pairings and twist factors displayed in \eqref{eq:finite-RT-general-revised}. The first Kirby move adjoins a distant $\pm1$--framed unknot and therefore multiplies the surgery sum by the Gauss factor $A_\pm(G,q)$; the chosen normalization cancels this contribution. The second Kirby move corresponds to an integral change of basis in the surgery presentation, and the sum is invariant under the induced relabelling of $G^m$ because $b_q$ is a bicharacter. Thus the expression is a closed $3$--manifold invariant.
\end{proof}

\begin{remark}\label{rem:fourier-S}
On the torus state space $\mathcal V^{\mathrm{RT}}_{\mathcal C(G,q)}(T^2,\lambda)\cong \mathbb C[G]$,
the normalized modular $S$--operator is the finite Fourier transform
\[
S(e_x)=|G|^{-1/2}\sum_{y\in G} b_q(x,y)\,e_y,
\]
up to the usual sign convention. In the cyclic case this is exactly the discrete Fourier transform used in the rank--one theory; compare \cite[Section 4.3 and Section~5.3]{Galviz1}.
\end{remark}

\medskip

We now pass from the finite modular setting to the broader quadratic surgery formalism on locally compact Abelian groups. This discussion is included for conceptual completeness and for comparison with the Abelian surgery integrals of Mattes--Polyak--Reshetikhin \cite{Mattes}; it is \emph{not} the input used later in the boundary TQFT comparison of Section \ref{section4}.

\begin{definition}\label{def:lca-quadratic-data}
Let $G$ be a locally compact Abelian group equipped with a fixed Haar measure $dg$, and
let $\theta:G\longrightarrow U(1)$ be a continuous quadratic function, meaning that
\begin{equation}\label{eq:Btheta-general-revised}
B_\theta(x,y):=\theta(x+y)\,\theta(x)^{-1}\theta(y)^{-1}
\end{equation}
is a continuous symmetric bicharacter on $G$.
For an oriented framed link $L=L_1\cup\cdots\cup L_m\subset S^3$ with linking matrix
$\mathcal L_L$, the corresponding \emph{quadratic Abelian surgery functional} is defined by
\begin{equation}\label{eq:raw-integral-general-revised}
Z^{\mathrm{raw}}_{(G,\theta)}(L)
:=
\int_{G^m}
\left(
\prod_{1\le i<j\le m} B_\theta(g_i,g_j)^{\mathcal L_{ij}}
\prod_{i=1}^m \theta(g_i)^{\mathcal L_{ii}}
\right)
\,d^m g,
\qquad d^m g:=dg_1\cdots dg_m,
\end{equation}
whenever the integral is defined in the ordinary or oscillatory sense.
\end{definition}

\begin{definition}\label{def:lca-kirby-normalization}
Let $U_\pm$ denote a distant $\pm1$--framed unknot. The associated one-component Gauss
factors are
\begin{equation}\label{eq:alpha-pm-general-revised}
\alpha_\pm(G,\theta):=\int_G \theta(s)^{\pm1}\,ds.
\end{equation}
If $\mathcal L_L$ has $m_\pm$ positive/negative eigenvalues in a stable diagonalization,
define
\begin{equation}\label{eq:Kirby-norm-general-revised}
\tau_{(G,\theta)}(M_L)
:=
\alpha_+(G,\theta)^{-m_+}\alpha_-(G,\theta)^{-m_-}
\,Z^{\mathrm{raw}}_{(G,\theta)}(L),
\end{equation}
provided all quantities exist.
\end{definition}

\begin{theorem}\label{thm:compact-abelian-surgery}
Let $G$ be a compact Abelian group equipped with Haar probability measure, and let
$\theta:G\to U(1)$ be a continuous quadratic function whose associated bicharacter
$B_\theta$ is nondegenerate. Assume moreover that
$\alpha_\pm(G,\theta)\neq 0$. Then the quantity
$\tau_{(G,\theta)}(M_L)$ of \eqref{eq:Kirby-norm-general-revised} is well defined and
depends only on the oriented homeomorphism type of the closed $3$--manifold $M_L$.
\end{theorem}

\begin{proof}
The proof is the usual Kirby-calculus argument. Isotopy invariance is built into the ribbon
evaluation. Under a handle slide, the linking matrix changes by an integral change of basis,
and the integrand in \eqref{eq:raw-integral-general-revised} transforms by the induced
automorphism of $G^m$; Haar measure is invariant under this change of variables because
$B_\theta$ is a bicharacter. Under a first Kirby move, adjoining a distant $\pm1$--framed
unknot multiplies the integral by $\alpha_\pm(G,\theta)$, and the normalization in
\eqref{eq:Kirby-norm-general-revised} cancels this factor. Hence
$\tau_{(G,\theta)}(M_L)$ is invariant under both Kirby moves.
\end{proof}

\begin{remark}\label{rem:normalization-warning}
There are two normalizations in play, and it is important not to confuse them.

\smallskip

\noindent
(i) For a \emph{finite} quadratic module $(G,q)$, the pointed RT invariant is the modular
category invariant \eqref{eq:finite-RT-general-revised}; this is the normalization used in
Sections~\ref{Surgery-formulas}--4.

\smallskip

\noindent
(ii) For a general \emph{compact or noncompact} Abelian group, the quantity
$\tau_{(G,\theta)}$ is an Abelian surgery functional defined from Haar integration.
In the finite case it coincides with the corresponding finite sum only after one fixes a
compatible choice of measure normalization. Since the later equivalence theorem uses only
the finite pointed modular category attached to $(G_K,q_K)$, we shall henceforth use the
RT normalization in the finite case and reserve the Haar-integral notation for the broader
compact and noncompact discussion.
\end{remark}

\begin{remark}\label{rem:noncompact-case}
If $G$ is noncompact, the integral in \eqref{eq:raw-integral-general-revised} is generally
oscillatory rather than absolutely convergent. In the model case $G=\mathbb R$ with
$\theta(x)=e^{ix^2}$, one recovers the Mattes--Polyak--Reshetikhin oscillatory Gaussian
integrals. When the surgery matrix is nondegenerate these are honest oscillatory Gaussians;
in degenerate situations one obtains distributions only under the usual kernel-compatibility
condition \cite{Mattes}.
\end{remark}

\subsection{Surgery formulas and finite quadratic modules}
\label{Surgery-formulas}

We now fix the finite quadratic data that control both the closed toral
Chern--Simons partition function and the closed Reshetikhin--Turaev invariant. It is the natural language in which the closed equivalence theorem will be stated.

Let \(M=M_L\) be a closed connected oriented \(3\)--manifold obtained by integral
surgery on an oriented framed \(m\)-component link
\(
L=L_1\cup\cdots\cup L_m\subset S^3
\),
with linking matrix
\begin{equation}\label{eq:linking-matrix-closed-prelim}
\mathcal L=(\mathcal L_{ij})\in M_m(\mathbb Z),
\qquad
\mathcal L_{ii}=\operatorname{fr}(L_i),
\qquad
\mathcal L_{ij}=\operatorname{lk}(L_i,L_j)\quad (i\neq j).
\end{equation}
We write
\begin{equation}\label{eq:sigma-rho-nu-closed-prelim}
\sigma=\sigma(\mathcal L),
\qquad
\rho=\operatorname{rank}(\mathcal L),
\qquad
\nu=\operatorname{null}(\mathcal L)=m-\rho=b_1(M).
\end{equation}
Choose an integral change of basis \(U\in GL_m(\mathbb Z)\) such that
\begin{equation}\label{eq:L-block-form-closed-prelim}
U^\top\mathcal L U
=
\begin{pmatrix}
\mathcal L_{\mathrm{reg}} & 0\\
0 & 0
\end{pmatrix},
\qquad
\mathcal L_{\mathrm{reg}}\in M_\rho(\mathbb Z),
\end{equation}
with \(\mathcal L_{\mathrm{reg}}\) nondegenerate over \(\mathbb Q\). Recall from Section~\ref{sec:3.1} that the even lattice \((\Lambda,K)\) determines the finite
quadratic module
\[
G_K=\Lambda^*/K\Lambda,
\qquad
q_K([x])=\exp\!\bigl(\pi i\,x^\top K^{-1}x\bigr),
\qquad
\Omega_K(a,b)=\frac{q_K(a+b)}{q_K(a)q_K(b)},
\]
and hence the pointed modular category \(\mathcal C(G_K,q_K)\); compare the standard pointed construction in Reshetikhin--Turaev theory \cite[Chapter~I--II]{Turaev1994}.

For an \(m\)-tuple
\(
a=(a_1,\dots,a_m)\in G_K^m
\),
choose lifts
\(
x_i\in \Lambda^*
\)
of the classes \(a_i\), and write
\(
x=(x_1,\dots,x_m)\in (\Lambda^*)^m
\).
We define
\begin{equation}\label{eq:Q-L-K-closed-prelim}
Q_{\mathcal L,K}(a)
\equiv
\frac12\,x^\top(\mathcal L\otimes K^{-1})x
\pmod 1.
\end{equation}
The evenness of \(K\) ensures that this quantity is independent of the chosen lifts.

\begin{lemma}\label{lem:QLK-well-defined-closed-prelim}
The class \(Q_{\mathcal L,K}(a)\in \mathbb Q/\mathbb Z\) defined by
\eqref{eq:Q-L-K-closed-prelim} is well defined, i.e.\ independent of the chosen lifts
\(x_i\in\Lambda^*\) of the elements \(a_i\in G_K\).
\end{lemma}

\begin{proof}
Suppose that \(x_i\) is replaced by \(x_i+K\lambda_i\), where
\(\lambda_i\in\Lambda\), and write
\(
\lambda=(\lambda_1,\dots,\lambda_m)\in\Lambda^m
\).
Then the corresponding change in
\(
\frac12\,x^\top(\mathcal L\otimes K^{-1})x
\)
is
\[
x^\top(\mathcal L\otimes I)\lambda
+
\frac12\,\lambda^\top(\mathcal L\otimes K)\lambda.
\]
The first term is integral because \(x_i\in\Lambda^*\), \(\lambda_j\in\Lambda\), and
\(\mathcal L_{ij}\in\mathbb Z\). Expanding the second term gives
\[
\frac12\,\lambda^\top(\mathcal L\otimes K)\lambda
=
\sum_{i<j}\mathcal L_{ij}\,K(\lambda_i,\lambda_j)
+
\frac12\sum_i \mathcal L_{ii}\,K(\lambda_i,\lambda_i).
\]
The off-diagonal contribution is integral because
\(K(\Lambda,\Lambda)\subset\mathbb Z\) and \(\mathcal L_{ij}\in\mathbb Z\).
The diagonal contribution is integral because \(K\) is even on \(\Lambda\), so
\(K(\lambda_i,\lambda_i)\in 2\mathbb Z\) for all \(\lambda_i\in\Lambda\).
Hence the total variation lies in \(\mathbb Z\), and
\(Q_{\mathcal L,K}(a)\) is well defined modulo \(1\).
\end{proof}

We now state the closed surgery formula for the pointed modular category
\(\mathcal C(G_K,q_K)\). In the general modular-category setting it is the standard Reshetikhin--Turaev surgery prescription, reduced here to a pointed
category, see for example \cite[Chapter~II, Chapter~IV]{Turaev1994}.

\begin{prop}\label{prop:RT-closed-surgery-GK}
Let \(M_L\) be presented by surgery on the framed link \(L\) with linking matrix
\(\mathcal L\), and write \(\sigma=\sigma(\mathcal L)\). Define the quadratic Gauss sums
\begin{equation}\label{eq:toral-closed-invt}
p_\pm(K):=\sum_{u\in G_K} q_K(u)^{\pm1}.
\end{equation}
Then the \emph{raw} closed Reshetikhin--Turaev surgery scalar associated with the pointed
modular category \(\mathcal C(G_K,q_K)\) is
\begin{equation}\label{eq:RT-closed-surgery-n}
Z^{\mathrm{RT,raw}}_{\mathcal C(G_K,q_K)}(M_L)
=
|G_K|^{-1/2}\,
p_+(K)^{\frac{-m-\sigma}{2}}\,
p_-(K)^{\frac{-m+\sigma}{2}}
\sum_{a\in G_K^m}
\exp\!\bigl(2\pi i\,Q_{\mathcal L,K}(a)\bigr).
\end{equation}
\end{prop}

\begin{proof}
Since \(\mathcal C(G_K,q_K)\) is pointed, every simple object has quantum dimension
\(1\). A coloring \(a=(a_1,\dots,a_m)\in G_K^m\) contributes a factor
\(q_K(a_i)^{\mathcal L_{ii}}\) from the framing on the component \(L_i\), and a factor
\(\Omega_K(a_i,a_j)^{\mathcal L_{ij}}\) from the linking of \(L_i\) and \(L_j\) for
\(i\neq j\). Therefore the RT evaluation of the colored surgery link is
\[
\prod_{1\le i<j\le m}\Omega_K(a_i,a_j)^{\mathcal L_{ij}}
\prod_{i=1}^m q_K(a_i)^{\mathcal L_{ii}}.
\]
If \(x_i\in\Lambda^*\) are lifts of \(a_i\), then by the definitions of \(q_K\) and
\(\Omega_K\) this product is
\[
\exp\!\Bigl(
\pi i\sum_i \mathcal L_{ii}\,x_i^\top K^{-1}x_i
+
2\pi i\sum_{i<j}\mathcal L_{ij}\,x_i^\top K^{-1}x_j
\Bigr)
=
\exp\!\bigl(2\pi i\,Q_{\mathcal L,K}(a)\bigr).
\]
Substituting this evaluation into the standard surgery formula for a pointed modular
category gives \eqref{eq:RT-closed-surgery-n}. The factors \(p_\pm(K)\) are the
one-component Gauss sums contributed by a distant \(\pm1\)-framed unknot, and hence
implement the usual Kirby normalization; compare the general RT surgery construction in
\cite[Chapter~IV]{Turaev1994}.
\end{proof}

Formula \eqref{eq:RT-closed-surgery-n} is thus the finite-quadratic-module surgery
expression that will be compared in Section \ref{section4} with the closed toral
Chern--Simons partition function. It is the precise higher-rank analogue of the cyclic
pointed RT surgery formula, with \((\mathbb Z_k,q_k)\) replaced by the discriminant
quadratic module \((G_K,q_K)\).

\subsection{Reshetikhin–Turaev boundary theory}

We now pass from the closed surgery formula to the
Reshetikhin--Turaev theory with boundary associated with the finite quadratic
module \((G_K,q_K)\) determined by the even lattice \((\Lambda,K)\).
Since \(\mathcal C(G_K,q_K)\) is a pointed modular category, Turaev's
extended RT construction applies and assigns state spaces to
extended surfaces and linear maps to extended bordisms; see
\cite[Section~IV.1--IV.3]{Turaev1994}. We record here the features of that
construction that will be used later. Recall that
\[
G_K=\Lambda^*/K\Lambda,
\qquad
q_K([x])=\exp\!\bigl(\pi i\,x^\top K^{-1}x\bigr),
\qquad
\Omega_K(a,b)=\frac{q_K(a+b)}{q_K(a)q_K(b)}.
\]
Let \(\mathcal C(G_K,q_K)\) denote the pointed modular category attached to
this finite quadratic module.

\begin{theorem}\label{thm:toral-pointed-extended-TQFT}
Let \((G_K,q_K)\) be the finite quadratic module associated with the even
integral lattice \((\Lambda,K)\). Then the pointed modular category
\(\mathcal C(G_K,q_K)\) determines, by Turaev's construction, an extended
\((2+1)\)-dimensional TQFT
\[
Z^{\mathrm{RT}}_{C(G_K,q_K)}:\Cob^{\mathrm{ext}}_{2+1}\longrightarrow \mathrm{Vect}_{\mathbb C}.
\]
For every extended surface \((\Sigma,\lambda)\), this theory assigns a
finite-dimensional vector space
\[\mathcal V^{\mathrm{RT}}_{C(G_K,q_K)}(\Sigma,\lambda):=
Z^{\mathrm{RT}}_{C(G_K,q_K)}(\Sigma,\lambda),
\]
and for every extended bordism
\[
M:(\Sigma_{\mathrm{in}},\lambda_{\mathrm{in}})
\longrightarrow
(\Sigma_{\mathrm{out}},\lambda_{\mathrm{out}})
\]
it assigns a linear map
\[
Z^{\mathrm{RT}}_{C(G_K,q_K)}(M):
\mathcal V^{\mathrm{RT}}_{C(G_K,q_K)}(\Sigma_{\mathrm{in}},\lambda_{\mathrm{in}})
\longrightarrow
\mathcal V^{\mathrm{RT}}_{C(G_K,q_K)}(\Sigma_{\mathrm{out}},\lambda_{\mathrm{out}}).
\]
Moreover, the underlying uncorrected bordism operators are computed by the
standard handlebody-pairing rule, while the extended functor
\(Z^{\mathrm{RT}}_{C(G_K,q_K)}\) incorporates the usual Maslov anomaly correction; see
\cite[Section~IV.4--IV.9]{Turaev1994}.
\end{theorem}

\begin{proof}
By Proposition~\ref{prop:pointed-category-general}, the category
\(\mathcal C(G_K,q_K)\) is modular. Turaev's extended
RT construction therefore applies and produces the stated
extended TQFT together with its Maslov-corrected gluing law; see
\cite[Section~IV.1--IV.3]{Turaev1994}.
\end{proof}

We begin with the torus, which will serve as the basic building block for the
general handlebody description.

\begin{prop}\label{prop:torus-state-space-toral}
Let \(\Sigma=T^2\) with Lagrangian
\[
\lambda=\langle \alpha\rangle \subset H_1(T^2;\mathbb R),
\]
where \(\alpha\) is the meridian class. Then there is a canonical
identification
\[
\mathcal V^{\mathrm{RT}}_{C(G_K,q_K)}(T^2,\lambda)\cong \mathbb C[G_K].
\]
More precisely, if \(H\) is the solid torus with
\(\ker(H_1(T^2;\mathbb R)\to H_1(H;\mathbb R))=\lambda\), then the classes of
core circles colored by \(a\in G_K\) form a basis
\[
\mathcal V^{\mathrm{RT}}_{C(G_K,q_K)}(T^2,\lambda)
=
\mathrm{Span}\{e_a\mid a\in G_K\}
\cong
\mathbb C[G_K].
\]
\end{prop}

\begin{proof}
This is the standard pointed-handlebody description of the torus state space. Let \(H\) be the solid torus whose
meridian bounds a disk, so that \(\partial H=T^2\) and the induced Lagrangian
is \(\lambda\). By Turaev's skein-theoretic construction
\cite[Section~IV.1--IV.2]{Turaev1994}, the state space
\(\mathcal V^{\mathrm{RT}}_{C(G_K,q_K)}(T^2,\lambda)\) is generated by
\(\mathcal C(G_K,q_K)\)-colored ribbon graphs in \(H\), modulo the usual local
relations.

Since \(\mathcal C(G_K,q_K)\) is pointed, every simple object is invertible and
every morphism space between simple tensor products is one-dimensional.
Therefore any colored ribbon graph in \(H\) is reduced, by isotopy and fusion,
to a disjoint union of parallel copies of the core circle. If two parallel core
circles are colored by \(a,b\in G_K\), the fusion rules identify them with a
single core circle colored by \(a+b\). It follows that every state is
represented by a single core colored by some \(a\in G_K\), and there are no
further linear relations among these generators. Hence
\[
\mathcal V^{\mathrm{RT}}_{C(G_K,q_K)}(T^2,\lambda)
=
\mathrm{Span}\{e_a\mid a\in G_K\}
\cong
\mathbb C[G_K].
\qedhere
\]
\end{proof}

The dual pairing on the torus state space is determined by the Hopf--link
evaluation.

\begin{prop}\label{prop:torus-pairing-toral}
Let \(\{e_a\}_{a\in G_K}\) be the basis of
\(\mathcal V^{\mathrm{RT}}_{C(G_K,q_K)}(T^2,\lambda)\) from
Proposition~\ref{prop:torus-state-space-toral}, and let
\(\{e_a^\vee\}_{a\in G_K}\subset \mathcal V^{\mathrm{RT}}_{C(G_K,q_K)}(T^2,\lambda)^*\) denote the
covectors represented by the oppositely oriented solid torus with core colored
by \(a\). Then
\[
\langle e_a^\vee,e_b\rangle
=
Z^{\mathrm{RT}}_{C(G_K,q_K)}(S^3,\text{Hopf link colored by }a,b)
=
S^{\mathrm{Hopf}}_{a,b}
=
\Omega_K(a,b).
\]
Equivalently, if \(\{\varepsilon_b\}_{b\in G_K}\) is the algebraic dual basis,
then
\[
e_a^\vee=\sum_{b\in G_K}\Omega_K(a,b)\,\varepsilon_b.
\]
In particular, the normalized modular \(S\)-operator on the torus state space is
\[
S(e_a)=|G_K|^{-1/2}\sum_{b\in G_K}\Omega_K(a,b)\,e_b,
\]
up to the conventional sign choice.
\end{prop}

\begin{proof}
By Turaev's construction of the pairing on the torus state space
\cite[Section~IV.1--IV.2]{Turaev1994}, one glues two solid tori along their
boundary by the diffeomorphism exchanging meridian and longitude. The resulting
closed $3$--manifold is \(S^3\). If the core circles are colored by \(a\) and
\(b\), they form a Hopf link in \(S^3\). Hence
\[
\langle e_a^\vee,e_b\rangle
=
Z^{\mathrm{RT}}_{C(G_K,q_K)}(S^3,\text{Hopf link colored by }a,b).
\]
In the pointed category \(\mathcal C(G_K,q_K)\), the Hopf--link matrix is the
bicharacter \(\Omega_K\), so
\[
\langle e_a^\vee,e_b\rangle=\Omega_K(a,b).
\]
The formula for \(e_a^\vee\) in the algebraic dual basis follows immediately,
and the displayed expression for the normalized modular \(S\)-operator is the
corresponding finite Fourier transform.
\end{proof}

For the comparison with toral Chern--Simons theory, one needs not only the torus
case but the general genus-\(g\) handlebody description.

\begin{prop}\label{prop:surface-basis-toral}
Let \((\Sigma,\lambda)\) be a connected extended surface of genus \(g\). Choose a
handlebody \(H\) with \(\partial H=\Sigma\) such that
\[
\ker(H_1(\Sigma;\mathbb R)\to H_1(H;\mathbb R))=\lambda.
\]
Then \(\mathcal V^{\mathrm{RT}}_{C(G_K,q_K)}(\Sigma,\lambda)\) admits a preferred basis indexed by
\(G_K^{\,g}\). More precisely, after choosing a standard system of \(g\) core
circles in \(H\), the vectors
\[
e_{\mathbf a},\qquad \mathbf a=(a_1,\dots,a_g)\in G_K^{\,g},
\]
obtained by coloring the \(i\)-th core by \(a_i\), form a basis. In particular,
\[
\dim \mathcal V^{\mathrm{RT}}_{C(G_K,q_K)}(\Sigma,\lambda)=|G_K|^{\,g}.
\]
\end{prop}

\begin{proof}
This is the higher-genus pointed-handlebody analogue of Proposition~\ref{prop:torus-state-space-toral}. By Turaev's handlebody model \cite[Section~IV.1--IV.2]{Turaev1994}, the space \(\mathcal V^{\mathrm{RT}}_{C(G_K,q_K)}(\Sigma,\lambda)\) is generated by colored ribbon graphs in \(H\). Since the category is pointed, fusion reduces every such graph to a standard graph supported on the \(g\) handlebody cores. The color on the \(i\)-th core is an element \(a_i\in G_K\), and the resulting vectors \(e_{\mathbf a}\) are linearly independent for the same reason as in the torus case: there are no nontrivial relations beyond the pointed fusion rules. Therefore the vectors indexed by \(G_K^{\,g}\) form a basis.
\end{proof}

We next compare the RT and toral state spaces in genus $g$, obtaining the canonical boundary identification used later in Section \ref{section4}.

\begin{theorem}\label{thm:toral-state-space-identification}
Let \((\Sigma,\lambda)\) be a connected extended surface of genus \(g\), and let $L_\lambda\subset H^1(\Sigma;\mathbb R)$ be the rational Lagrangian corresponding to \(\lambda\subset H_1(\Sigma;\mathbb R)\)
under the boundary Lagrangian correspondence of Lemma \ref{lem:lambdaX-LX}.
Then the Reshetikhin--Turaev state space $\mathcal V^{\mathrm{RT}}_{\mathcal C(G_K,q_K)}(\Sigma,\lambda)$
and the toral Chern--Simons state space $\mathcal H_{\mathbb T,K}(\Sigma,L_\lambda)$ have the same dimension and admit preferred bases indexed by \(G_K^{\,g}\). In particular,
after fixing the standard handlebody presentation adapted to \(\lambda\), one obtains a
canonical vector-space isomorphism
\begin{equation}\label{eq:PhiSigmaK}
\Phi_{\Sigma,K}:
\mathcal V^{\mathrm{RT}}_{\mathcal C(G_K,q_K)}(\Sigma,\lambda)
\xrightarrow{\ \cong\ }
\mathcal H_{\mathbb T,K}(\Sigma,L_\lambda),
\end{equation}
sending the preferred RT handlebody basis to the preferred Bohr--Sommerfeld basis.
Moreover,
\[
\dim \mathcal V^{\mathrm{RT}}_{\mathcal C(G_K,q_K)}(\Sigma,\lambda)
=
\dim \mathcal H_{\mathbb T,K}(\Sigma,L_\lambda)
=
|G_K|^{\,g}
=
|\det K|^{\,g}.
\]
\end{theorem}

\begin{proof}
On the toral Chern--Simons side, Section~\ref{sec:toral-CS} identifies the boundary phase space with the
symplectic torus
\[
\mathcal M_\Sigma(\mathbb T)
=
H^1(\Sigma;\mathfrak t)\big/H^1(\Sigma;\Lambda),
\]
equipped with the real polarization determined by \(L_\lambda\). By the geometric
quantization results recalled there, the Bohr--Sommerfeld leaves form a torsor for
\(G_K^{\,g}\), and each Bohr--Sommerfeld leaf contributes a one-dimensional summand to
the quantization. Hence
\[
\dim \mathcal H_{\mathbb T,K}(\Sigma,L_\lambda)=|G_K|^{\,g};
\]
see Proposition~\ref{prop:BS-leaves} and Theorem~\ref{thm:toral-quantization-dimension}. On the Reshetikhin--Turaev side, let \(H\) be a handlebody with
\(\partial H=\Sigma\) and
\[
\ker\!\bigl(H_1(\Sigma;\mathbb R)\to H_1(H;\mathbb R)\bigr)=\lambda.
\]
By Proposition~\ref{prop:surface-basis-toral}, the state space $\mathcal V^{\mathrm{RT}}_{\mathcal C(G_K,q_K)}(\Sigma,\lambda)$ admits a preferred basis $\{e_{\mathbf a}\}_{\mathbf a\in G_K^{\,g}}$ obtained by coloring the \(g\) standard handlebody cores by labels
\(\mathbf a=(a_1,\dots,a_g)\in G_K^{\,g}\). In particular,
\[
\dim \mathcal V^{\mathrm{RT}}_{\mathcal C(G_K,q_K)}(\Sigma,\lambda)=|G_K|^{\,g}.
\]

Thus both state spaces have the same dimension and are equipped with preferred bases
indexed by the same finite set \(G_K^{\,g}\). Sending the RT basis vector
\(e_{\mathbf a}\) to the toral Bohr--Sommerfeld basis vector indexed by the same
\(\mathbf a\in G_K^{\,g}\) defines the isomorphism \eqref{eq:PhiSigmaK}. The compatibility
of these identifications with bordism operators is proved later in Section~\ref{section4}.
\end{proof}

We now record the standard closure rule for matrix coefficients of bordism operators.

\begin{prop}\label{prop:bordism-matrix-coefficients-toral}
Let $M:(\Sigma_{\mathrm{in}},\lambda_{\mathrm{in}})
\to
(\Sigma_{\mathrm{out}},\lambda_{\mathrm{out}})$
be an extended bordism, and choose the preferred handlebody bases
\[
\{e_{\mathbf a}^{\mathrm{in}}\}_{\mathbf a\in G_K^{g_{\mathrm{in}}}}
\subset
\mathcal V^{\mathrm{RT}}_{C(G_K,q_K)}(\Sigma_{\mathrm{in}},\lambda_{\mathrm{in}}),
\qquad
\{(e_{\mathbf b}^{\mathrm{out}})^\vee\}_{\mathbf b\in G_K^{g_{\mathrm{out}}}}
\subset
\mathcal V^{\mathrm{RT}}_{C(G_K,q_K)}(\Sigma_{\mathrm{out}},\lambda_{\mathrm{out}})^*.
\]
Let \(M_{\mathbf a,\mathbf b}\) be the closed extended $3$--manifold obtained by
gluing to \(M\) the incoming handlebody colored by \(\mathbf a\) and the
outgoing handlebody colored by \(\mathbf b\). Then
\begin{equation}\label{eq:matrix-coefficients-by-closure}
\big\langle (e_{\mathbf b}^{\mathrm{out}})^\vee,
Z^{\mathrm{RT}}_{C(G_K,q_K)}(M)\,e_{\mathbf a}^{\mathrm{in}}\big\rangle
=
Z^{\mathrm{RT}}_{C(G_K,q_K)}(M_{\mathbf a,\mathbf b}).
\end{equation}
\end{prop}

\begin{proof}
Equation \eqref{eq:matrix-coefficients-by-closure} is the standard pairing rule
in Turaev's extended RT formalism; see \cite[Section~IV.1--IV.3]{Turaev1994}. By definition, the covector
\((e_{\mathbf b}^{\mathrm{out}})^\vee\) is represented by the oppositely
oriented outgoing handlebody colored by \(\mathbf b\), while
\(e_{\mathbf a}^{\mathrm{in}}\) is represented by the incoming handlebody
colored by \(\mathbf a\). Gluing these standard representatives to the bordism
\(M\) yields the closed extended manifold \(M_{\mathbf a,\mathbf b}\), and the
pairing is, by construction, the RT invariant of this closure.
\end{proof}

As a consequence, the surgery formula of Section~\ref{Surgery-formulas} gives explicit expressions
for matrix coefficients whenever the closure \(M_{\mathbf a,\mathbf b}\) is
presented by surgery.

\begin{corollary}\label{cor:bordism-surgery-coefficients-toral}
Assume that the closure \(M_{\mathbf a,\mathbf b}\) is presented by surgery on a
framed link with linking matrix \(\mathcal L_{\mathbf a,\mathbf b}\). Then
\[
\big\langle (e_{\mathbf b}^{\mathrm{out}})^\vee,
Z^{\mathrm{RT}}_{C(G_K,q_K)}(M)\,e_{\mathbf a}^{\mathrm{in}}\big\rangle
=
|G_K|^{-1/2}\,
p_+(K)^{\frac{-m-\sigma}{2}}\,
p_-(K)^{\frac{-m+\sigma}{2}}
\sum_{u\in G_K^m}
\exp\!\bigl(2\pi i\,Q_{\mathcal L_{\mathbf a,\mathbf b},K}(u)\bigr),
\]
where \(m\) is the number of surgery components and
\(\sigma=\sigma(\mathcal L_{\mathbf a,\mathbf b})\).
\end{corollary}

\begin{proof}
Apply Proposition~\ref{prop:bordism-matrix-coefficients-toral} and then
Proposition~\ref{prop:RT-closed-surgery-GK}.
\end{proof}

Finally, we recall the functoriality statement in the form needed later.

\begin{prop}\label{prop:rt-gluing-law-toral}
Let
\[
M_1:(\Sigma_0,\lambda_0)\to(\Sigma_1,\lambda_1),
\qquad
M_2:(\Sigma_1,\lambda_1)\to(\Sigma_2,\lambda_2)
\]
be extended bordisms. Then the Maslov-corrected RT operators satisfy the strict
gluing law
\[
Z^{\mathrm{RT}}_{C(G_K,q_K)}(M_2\circ M_1)
=
Z^{\mathrm{RT}}_{C(G_K,q_K)}(M_2)\circ Z^{\mathrm{RT}}_{C(G_K,q_K)}(M_1).
\]
Equivalently, before Maslov anomaly normalization the underlying operators satisfy the
usual projective composition law with Maslov anomaly; see
\cite[Section~IV.3]{Turaev1994}.
\end{prop}

\begin{proof}
This is the standard functoriality statement of Turaev's extended
Reshetikhin--Turaev theory; see \cite[Section~IV.8–IV.9]{Turaev1994}. In the present
paper we work with the Maslov-corrected extended theory, so the corrected
operators compose strictly.
\end{proof}

Propositions~\ref{prop:torus-state-space-toral}--\ref{prop:rt-gluing-law-toral} give the RT boundary theory in the form needed
for Section~\ref{section4}: preferred bases indexed by \(G_K^{\,g}\), explicit torus
pairings determined by the bicharacter \(\Omega_K\), and matrix coefficients
computed by closure with colored handlebodies. This is exactly the structure
that will be compared with the toral Chern--Simons boundary theory.

\section{Equivalence of Toral Chern--Simons and Reshetikhin–Turaev Theories}
\label{section4}
\subsection{\texorpdfstring{Closed equivalence for \(\mathbb T\cong U(1)^n\)}{Closed equivalence for T isomorphic to U(1)n}}
\label{subsec:toral-closed-comparison}

We now prove the closed equivalence theorem for the toral theory. We first rewrite the toral Chern--Simons partition function in a form determined by the quadratic refinement of the torsion linking pairing. We then show that the raw Reshetikhin--Turaev scalar associated with the finite quadratic module \((G_K,q_K)\) is given by the same exact finite quadratic Gauss sum.

Throughout this subsection, $\mathbb T=\mathfrak t/\Lambda\cong U(1)^n$ is a compact torus, and $K:\Lambda\times\Lambda\to \mathbb Z$ is an even, integral, nondegenerate symmetric bilinear form. After choosing a \(\Lambda\)-basis, we identify \(K\) with an integral symmetric matrix
\[
K\in M_n(\mathbb Z),\qquad K^\top=K,\qquad \det K\neq 0.
\]
Let
\[
G_K:=\Lambda^*/K\Lambda,
\qquad
|G_K|=|\det K|,
\qquad
q_K([u])=\exp\!\bigl(\pi i\,u^\top K^{-1}u\bigr).
\]

Let \(M=M_L\) be presented by integral surgery on an oriented framed
\(m\)-component link \(L\subset S^3\) with linking matrix \(\mathcal L\).
Write
\[
\sigma=\sigma(\mathcal L),
\qquad
\nu=b_1(M),
\qquad
\rho=m-\nu.
\]
Choose an integral change of basis \(U\in GL_m(\mathbb Z)\) such that
\[
U^\top \mathcal L U=
\begin{pmatrix}
\mathcal L_{\mathrm{reg}}&0\\
0&0
\end{pmatrix},
\qquad
\mathcal L_{\mathrm{reg}}\in M_\rho(\mathbb Z),
\]
with \(\mathcal L_{\mathrm{reg}}\) nondegenerate over \(\mathbb Q\). Then
\begin{equation}\label{eq:toral-torsion-regular-block}
\Tors H_1(M;\mathbb Z)\cong
\mathbb Z^\rho/\mathcal L_{\mathrm{reg}}\mathbb Z^\rho,
\qquad
\bigl|\Tors H_1(M;\mathbb Z)\bigr|
=
\bigl|\det(\mathcal L_{\mathrm{reg}})\bigr|.
\end{equation}
We also recall that
\begin{equation}\label{eq:toral-mM}
m_M=\frac12\bigl(b_1(M)-1\bigr)=\frac12(\nu-1),
\end{equation}
in the normalization of the closed toral partition function Proposition \ref{thm:toral-boundary-vector}.

The proof has three ingredients: first, the value of the torsion half-density integral on the identity component of the
flat moduli space. Second,  the description of the toral Chern--Simons phases on torsion components as a finite quadratic Gauss sum. Third, the reciprocity formula converting the RT surgery scalar into the same finite quadratic Gauss sum.

\begin{lemma}\label{lem:toral-torsion-integral}
Let \(M\) be a closed connected oriented \(3\)--manifold. Then the
translation-invariant half-density \((T_M)^{1/2}\) appearing in Proposition \ref{thm:toral-boundary-vector}  satisfies
\begin{equation}\label{eq:toral-torsion-integral}
\int_{\mathcal M_M(\mathbb T)_0}(T_M)^{1/2}
=
\bigl|\Tors H_1(M;\mathbb Z)\bigr|^{n/2}
=
\bigl|\det(\mathcal L_{\mathrm{reg}})\bigr|^{n/2},
\end{equation}
where
\[
\mathcal M_M(\mathbb T)_0
\cong
H^1(M;\mathfrak t)/H^1(M;\Lambda)
\]
is the identity component of the moduli space of flat \(\mathbb T\)-connections.
\end{lemma}

\begin{proof}
Since \(\mathfrak t\cong\mathbb R^n\) is a trivial real local system, $C^\bullet(M;\mathfrak t)\cong C^\bullet(M;\mathbb R)^{\oplus n}.$ Reidemeister torsion is multiplicative under direct sums of coefficient systems, hence $T_M(\mathfrak t)=T_M(\mathbb R)^{\otimes n}.$ Passing to positive square roots, the half-density on
\[
\mathcal M_M(\mathbb T)_0
\cong
\bigl(H^1(M;\mathbb R)/H^1(M;\mathbb Z)\bigr)^n
\]
is the exterior product of the \(n\) rank--one half-densities. Therefore
\[
\int_{\mathcal M_M(\mathbb T)_0}(T_M)^{1/2}
=
\left(
\int_{\mathcal M_M(U(1))_0}(T_M)^{1/2}
\right)^n.
\]
By the rank--one formula proved in \cite[Theorem~5.2]{Galviz1},
\[
\int_{\mathcal M_M(U(1))_0}(T_M)^{1/2}
=
\bigl|\Tors H_1(M;\mathbb Z)\bigr|^{1/2}.
\]
Hence
\[
\int_{\mathcal M_M(\mathbb T)_0}(T_M)^{1/2}
=
\bigl|\Tors H_1(M;\mathbb Z)\bigr|^{n/2}.
\]
The second equality in \eqref{eq:toral-torsion-integral} follows from
\eqref{eq:toral-torsion-regular-block}.
\end{proof}

Fix once and for all the chosen basis of \(\Lambda\), so that \(\Lambda\cong\mathbb Z^n\).
Then
\[
H^2(M;\Lambda)\cong H^2(M;\mathbb Z^n)\cong H^2(M;\mathbb Z)^n,
\]
and therefore $\Tors H^2(M;\Lambda)\cong \bigl(\Tors H^2(M;\mathbb Z)\bigr)^n.$ Using also \[\Tors H^2(M;\mathbb Z)\cong \Tors H_1(M;\mathbb Z),\]
we obtain
\begin{equation}\label{eq:toral-torsion-sectors-identification}
\Tors H^2(M;\Lambda)
\cong
\mathbb Z^{\rho n}/(\mathcal L_{\mathrm{reg}}\otimes I_n)\mathbb Z^{\rho n}.
\end{equation}

\begin{lemma}\label{lem:toral-quadratic-refinement}
Define
\begin{equation}\label{eq:toral-q-LK}
q_{\mathcal L,K}([x])
\equiv
\frac12\,x^\top(\mathcal L_{\mathrm{reg}}^{-1}\otimes K)x
\pmod 1,
\qquad
[x]\in
\mathbb Z^{\rho n}/(\mathcal L_{\mathrm{reg}}\otimes I_n)\mathbb Z^{\rho n}.
\end{equation}
Then \(q_{\mathcal L,K}\) is well defined.
\end{lemma}

\begin{proof}
Let \(x'=x+(\mathcal L_{\mathrm{reg}}\otimes I_n)z\), where \(z\in\mathbb Z^{\rho n}\).
Then
\[
\frac12\,x'^\top(\mathcal L_{\mathrm{reg}}^{-1}\otimes K)x'
-
\frac12\,x^\top(\mathcal L_{\mathrm{reg}}^{-1}\otimes K)x
=
x^\top(I_\rho\otimes K)z
+
\frac12\,z^\top(\mathcal L_{\mathrm{reg}}\otimes K)z.
\]
The first term is integral because \(x,z\in\mathbb Z^{\rho n}\) and \(K\) is integral.
Write \(z=(z_1,\dots,z_\rho)\) with \(z_i\in\mathbb Z^n\). Then
\[
\frac12\,z^\top(\mathcal L_{\mathrm{reg}}\otimes K)z
=
\frac12\sum_i (\mathcal L_{\mathrm{reg}})_{ii}\,z_i^\top K z_i
+
\sum_{i<j}(\mathcal L_{\mathrm{reg}})_{ij}\,z_i^\top K z_j.
\]
Since \(K\) is even, each \(z_i^\top K z_i\in 2\mathbb Z\), so the diagonal contribution is
integral; the off-diagonal contribution is integral because \(K\) is integral. Hence the
difference lies in \(\mathbb Z\), proving that \(q_{\mathcal L,K}\) is well defined modulo \(1\).
\end{proof}

\begin{prop}\label{prop:toral-CS-normal-form}
With the normalization of the closed toral partition function in Proposition \ref{thm:toral-boundary-vector}, one has
\begin{equation}\label{eq:toral-CS-geometric}
Z^{\mathrm{CS,raw}}_{\mathbb T,K}(M)
=
\frac{|G_K|^{m_M}}{\#\,\Tors H^2(M;\Lambda)}
\sum_{p\in\Tors H^2(M;\Lambda)}
\int_{\mathcal M_{M,p}(\mathbb T)}
\sigma_{M,p}\,(T_M)^{1/2},
\end{equation}
and this expression reduces to the finite quadratic Gauss sum
\begin{equation}\label{eq:toral-CS-torsion-normal-form}
Z^{\mathrm{CS,raw}}_{\mathbb T,K}(M)
=
|G_K|^{m_M}\,
|\det(\mathcal L_{\mathrm{reg}})|^{-n/2}
\sum_{[x]\in \mathbb Z^{\rho n}/(\mathcal L_{\mathrm{reg}}\otimes I_n)\mathbb Z^{\rho n}}
\exp\!\Bigl(
\pi i\,x^\top(\mathcal L_{\mathrm{reg}}^{-1}\otimes K)x
\Bigr).
\end{equation}
\end{prop}

\begin{proof}
Equation \eqref{eq:toral-CS-geometric} is exactly the closed toral partition function of equation \eqref{eq:toral-closed-partition-function}. For closed \(M\), each connected component
\(\mathcal M_{M,p}(\mathbb T)\) is an affine translate of the identity component
\[
\mathcal M_M(\mathbb T)_0\cong H^1(M;\mathfrak t)/H^1(M;\Lambda),
\]
the half-density \((T_M)^{1/2}\) is translation invariant, and \(\sigma_{M,p}\) is constant on
the component indexed by \(p\). Hence
\[
Z^{\mathrm{CS,raw}}_{\mathbb T,K}(M)
=
\frac{|G_K|^{m_M}}{\#\,\Tors H^2(M;\Lambda)}
\left(
\sum_{p\in\Tors H^2(M;\Lambda)}\sigma_{M,p}
\right)
\int_{\mathcal M_M(\mathbb T)_0}(T_M)^{1/2}.
\]
By Lemma~\ref{lem:toral-torsion-integral},
\[
\int_{\mathcal M_M(\mathbb T)_0}(T_M)^{1/2}
=
|\det(\mathcal L_{\mathrm{reg}})|^{n/2}.
\]
Moreover, by the toral analogue of the rank--one identification
\cite[Theorem~5.3]{Galviz1}, the torsion component indexed by
\(
p\in\Tors H^2(M;\Lambda)
\)
contributes the phase
\[
\sigma_{M,p}=\exp\!\bigl(2\pi i\,q_{\mathcal L,K}(p)\bigr),
\]
so that, under the identification
\eqref{eq:toral-torsion-sectors-identification},
\[
\sum_{p\in\Tors H^2(M;\Lambda)}\sigma_{M,p}
=
\sum_{[x]\in \mathbb Z^{\rho n}/(\mathcal L_{\mathrm{reg}}\otimes I_n)\mathbb Z^{\rho n}}
\exp\!\Bigl(
\pi i\,x^\top(\mathcal L_{\mathrm{reg}}^{-1}\otimes K)x
\Bigr).
\]
Substituting these two identities into \eqref{eq:toral-CS-geometric} gives
\eqref{eq:toral-CS-torsion-normal-form}.
\end{proof}

For the pointed modular category \(\mathcal C(G_K,q_K)\), write
\begin{equation}\label{eq:toral-pm-gauss}
p_\pm(K):=\sum_{u\in G_K} q_K(u)^{\pm1}.
\end{equation}

\begin{lemma}\label{lem:toral-pm-evaluation}
One has
\begin{equation}\label{eq:toral-pplus-eval}
p_+(K)=e^{\frac{\pi i}{4}\sigma(K)}\,|G_K|^{1/2},
\qquad
p_-(K)=e^{-\frac{\pi i}{4}\sigma(K)}\,|G_K|^{1/2}.
\end{equation}
\end{lemma}

\begin{proof}
The finite quadratic module \((G_K,q_K)\) is the discriminant form of the even
nondegenerate lattice \((\Lambda,K)\). Milgram's formula therefore gives
\[
|G_K|^{-1/2}\sum_{u\in G_K} q_K(u)
=
e^{\frac{\pi i}{4}\sigma(K)};
\]
see, for example, \cite[Appendix 4]{milnor2013}. This is the first identity in
\eqref{eq:toral-pplus-eval}. The second follows by complex conjugation.
\end{proof}

\begin{lemma}\label{lem:toral-RT-null-separation}
Let \(Q_{\mathcal L,K}\) be the quadratic form from the RT surgery formula of Section~\ref{Surgery-formulas}. Then, after separating the \(\nu=b_1(M)\) null directions of \(\mathcal L\),
one has
\begin{equation}\label{eq:toral-RT-null-separated}
\sum_{a\in G_K^m}\exp\!\bigl(2\pi i\,Q_{\mathcal L,K}(a)\bigr)
=
|G_K|^\nu
\sum_{[y]\in G_K^\rho}
\exp\!\Bigl(
\pi i\,y^\top(\mathcal L_{\mathrm{reg}}\otimes K^{-1})y
\Bigr).
\end{equation}
\end{lemma}

\begin{proof}
Choose \(U\in GL_m(\mathbb Z)\) such that
\[
U^\top\mathcal L U=
\begin{pmatrix}
\mathcal L_{\mathrm{reg}}&0\\
0&0
\end{pmatrix}.
\]
Because \(U\) induces a bijection of the finite set \(G_K^m\), the RT sum is unchanged
after the corresponding change of variables. In the new coordinates write
\(a=(y,z)\) with \(y\in G_K^\rho\) and \(z\in G_K^\nu\). The quadratic form depends only
on the regular block, so the \(z\)-variables contribute the cardinality factor \(|G_K|^\nu\).
This gives \eqref{eq:toral-RT-null-separated}.
\end{proof}

\begin{lemma}\label{lem:toral-quadratic-reciprocity}
Let \(A\in M_\rho(\mathbb Z)\) be symmetric and nondegenerate, and let
\(K\in M_n(\mathbb Z)\) be symmetric, nondegenerate, and even. Then
\begin{align}
\label{eq:toral-block-reciprocity}
&\sum_{[y]\in \mathbb Z^{\rho n}/(I_\rho\otimes K)\mathbb Z^{\rho n}}
\exp\!\bigl(\pi i\,y^\top(A\otimes K^{-1})y\bigr) \\
&\qquad=
e^{\frac{\pi i}{4}\sigma(A\otimes K)}\,
|G_K|^{\rho/2}\,
|\det A|^{-n/2}
\sum_{[x]\in \mathbb Z^{\rho n}/(A\otimes I_n)\mathbb Z^{\rho n}}
\exp\!\bigl(\pi i\,x^\top(A^{-1}\otimes K)x\bigr).
\nonumber
\end{align}
\end{lemma}

\begin{proof}
This is the Deloup--Turaev reciprocity formula written in the form needed here.
Indeed, \cite[Theorem~1]{deloup2005} applied to the commuting symmetric forms
\[
A\otimes I_n
\qquad\text{and}\qquad
I_\rho\otimes K
\]
gives
\[
\begin{aligned}
|\det A|^{-n/2}
\sum_{[x]\in \mathbb Z^{\rho n}/(A\otimes I_n)\mathbb Z^{\rho n}}
\exp\!\bigl(\pi i\,x^\top(A^{-1}\otimes K)x\bigr)
&= \\
e^{-\frac{\pi i}{4}\sigma(A\otimes K)}\,
|G_K|^{-\rho/2}
\sum_{[y]\in \mathbb Z^{\rho n}/(I_\rho\otimes K)\mathbb Z^{\rho n}}
\exp\!\bigl(\pi i\,y^\top(A\otimes K^{-1})y\bigr).
&
\end{aligned}
\]
Rearranging gives \eqref{eq:toral-block-reciprocity}.
\end{proof}

For the equivalence theorem, we write \(Z^{\mathrm{RT,raw}}_{\mathcal C(G_K,q_K)}(M_L)\) for the uncorrected RT surgery scalar of Section~\ref{Surgery-formulas}. The Maslov anomaly correction for the scalar enters only later, in Section \ref{sec:extended-equivalence}.

\begin{prop}\label{prop:toral-RT-reciprocity-form}
Let \(M=M_L\) be presented by surgery on the framed link \(L\) with linking matrix
\(\mathcal L\). Then
\begin{equation}\label{eq:toral-RT-reciprocity-form}
\begin{split}
Z^{\mathrm{RT,raw}}_{\mathcal C(G_K,q_K)}(M_L)
&=
|G_K|^{m_M}\,
|\det(\mathcal L_{\mathrm{reg}})|^{-n/2} \\
&\times
\sum_{[x]\in \mathbb Z^{\rho n}/(\mathcal L_{\mathrm{reg}}\otimes I_n)\mathbb Z^{\rho n}}
\exp\!\Bigl(
\pi i\,x^\top(\mathcal L_{\mathrm{reg}}^{-1}\otimes K)x
\Bigr).
\end{split}
\end{equation}
\end{prop}

\begin{proof}
By the surgery formula of Section~\ref{Surgery-formulas},
\[
Z^{\mathrm{RT,raw}}_{\mathcal C(G_K,q_K)}(M_L)
=
|G_K|^{-1/2}\,
p_+(K)^{\frac{-m-\sigma}{2}}\,
p_-(K)^{\frac{-m+\sigma}{2}}
\sum_{a\in G_K^m}\exp\!\bigl(2\pi i\,Q_{\mathcal L,K}(a)\bigr).
\]
By Lemma~\ref{lem:toral-pm-evaluation},
\[
|G_K|^{-1/2}\,
p_+(K)^{\frac{-m-\sigma}{2}}\,
p_-(K)^{\frac{-m+\sigma}{2}}
=
e^{-\frac{\pi i}{4}\sigma(\mathcal L_{\mathrm{reg}}\otimes K)}\,
|G_K|^{-(m+1)/2},
\]
since \(\sigma(\mathcal L)=\sigma(\mathcal L_{\mathrm{reg}})\) and
\(\sigma(\mathcal L_{\mathrm{reg}})\sigma(K)=\sigma(\mathcal L_{\mathrm{reg}}\otimes K)\).
Now apply Lemma~\ref{lem:toral-RT-null-separation} and then
Lemma~\ref{lem:toral-quadratic-reciprocity} with \(A=\mathcal L_{\mathrm{reg}}\). This yields
\[
\begin{aligned}
Z^{\mathrm{RT,raw}}_{\mathcal C(G_K,q_K)}(M_L)
&=
|G_K|^{-(m+1)/2+\nu+\rho/2}\,
|\det(\mathcal L_{\mathrm{reg}})|^{-n/2}\\
&\times\sum_{[x]\in \mathbb Z^{\rho n}/(\mathcal L_{\mathrm{reg}}\otimes I_n)\mathbb Z^{\rho n}}
\exp\!\Bigl(
\pi i\,x^\top(\mathcal L_{\mathrm{reg}}^{-1}\otimes K)x
\Bigr).
\end{aligned}
\]
Finally, $m_M=-(m+1)/2+\nu+\rho/2$, because \(m=\rho+\nu\) and \(m_M=(\nu-1)/2\). This proves
\eqref{eq:toral-RT-reciprocity-form}.
\end{proof}

We may now state the closed equivalence theorem. Recall that  the Abelian Chern--Simons partition function is given by a normalized sum over torsion sectors, while the RT scalar is given by the same finite quadratic Gauss sum.

\begin{theorem}\label{thm:toral-closed-equivalence}
Let \(M=M_L\) be a closed connected oriented \(3\)--manifold presented by integral surgery
on a framed link \(L\subset S^3\) with linking matrix \(\mathcal L\). Let
\(\mathbb T=\mathfrak t/\Lambda\cong U(1)^n\), let
\(K:\Lambda\times\Lambda\to\mathbb Z\) be even, integral, symmetric, and nondegenerate,
and let \((G_K,q_K)\) be the associated finite quadratic module. Then
\begin{equation}\label{eq:toral-closed-comparison}
Z^{\mathrm{RT,raw}}_{\mathcal C(G_K,q_K)}(M_L)=Z^{\mathrm{CS,raw}}_{\mathbb T,K}(M_L).
\end{equation}
Equivalently,
\begin{align}
\label{eq:toral-CS-RT-common-gausssum}
Z^{\mathrm{CS,raw}}_{\mathbb T,K}(M_L)
&=
|G_K|^{m_M}\,
|\det(\mathcal L_{\mathrm{reg}})|^{-n/2} \\
&\qquad\times
\sum_{[x]\in \mathbb Z^{\rho n}/(\mathcal L_{\mathrm{reg}}\otimes I_n)\mathbb Z^{\rho n}}
\exp\!\Bigl(
\pi i\,x^\top(\mathcal L_{\mathrm{reg}}^{-1}\otimes K)x
\Bigr),
\nonumber
\end{align}
\end{theorem}

\begin{proof}
Equation \eqref{eq:toral-CS-RT-common-gausssum} is exactly Proposition~\ref{prop:toral-CS-normal-form}, while Proposition~\ref{prop:toral-RT-reciprocity-form} identifies the RT scalar with the same Gauss sum. Comparing the two formulas gives \eqref{eq:toral-closed-comparison}.
\end{proof}

\subsection{\texorpdfstring{Boundary equivalence for $U(1)^n$}{Boundary equivalence for U(1)n}}
\label{subsec:toral-boundary-equivalence}

We now prove the boundary analogue of the closed equivalence theorem. The essential point is that the bordism operator in each
theory is determined by its matrix elements against canonical handlebody states, and those
matrix elements are obtained by closing the bordism with labeled handlebodies and
evaluating the resulting closed extended \(3\)--manifold; compare Turaev's handlebody
pairing construction \cite[Section~IV.1--IV.3]{Turaev1994}.

It is important to note that the closed equivalence theorem by itself does \emph{not} immediately imply equality of boundary operators. Indeed, for a bordism \( X:(\Sigma_{\mathrm{in}},\lambda_{\mathrm{in}}) \to (\Sigma_{\mathrm{out}},\lambda_{\mathrm{out}}), \) the relevant matrix coefficients are computed by closing \(X\) with canonical incoming and outgoing handlebodies, namely through the closed manifolds \[ M_{v,w}:= H_{\mathrm{out}}(w)\cup_{\Sigma_{\mathrm{out}}}X\cup_{\Sigma_{\mathrm{in}}}H_{\mathrm{in}}(v). \] Thus one must work with the Maslov anomaly correction for the extended theories, not with the raw surgery operators, because bordism operators are defined through these closures and strict functoriality is recovered only in the extended formalism. Let \(\mathcal C(G_K,q_K)\) be the pointed modular category attached to the finite quadratic module \((G_K,q_K)\). Since \(\mathcal C(G_K,q_K)\) is modular, Turaev's extended Reshetikhin--Turaev construction applies.

Let $X:(\Sigma_{\mathrm{in}},\lambda_{\mathrm{in}})
\to
(\Sigma_{\mathrm{out}},\lambda_{\mathrm{out}})$ be an extended bordism, and let
\[
L_{\lambda_{\mathrm{in}}}\subset H^1(\Sigma_{\mathrm{in}};\mathbb R),
\qquad
L_{\lambda_{\mathrm{out}}}\subset H^1(\Sigma_{\mathrm{out}};\mathbb R)
\]
be the corresponding real polarizations on the toral Chern--Simons side, identified from
the Lagrangian data by Lemma \ref{lem:lambdaX-LX}. Denote by
\begin{equation}\label{eq:toral-RT-operator}
 Z^{\mathrm{RT}}_{\mathcal C(G_K,q_K)}(X)
\in
\Hom\!\Bigl(
\mathcal V^{\mathrm{RT}}_{\mathcal C(G_K,q_K)}(\Sigma_{\mathrm{in}},\lambda_{\mathrm{in}}),
\mathcal V^{\mathrm{RT}}_{\mathcal C(G_K,q_K)}(\Sigma_{\mathrm{out}},\lambda_{\mathrm{out}})
\Bigr)
\end{equation}
the Maslov--corrected Reshetikhin--Turaev operator, and by
\begin{equation}\label{eq:toral-CS-operator}
Z^{\mathrm{CS}}_{\mathbb T,K}(X)
\in
\Hom\!\Bigl(
\mathcal H_{\mathbb T,K}(\Sigma_{\mathrm{in}},L_{\lambda_{\mathrm{in}}}),
\mathcal H_{\mathbb T,K}(\Sigma_{\mathrm{out}},L_{\lambda_{\mathrm{out}}})
\Bigr)
\end{equation}
the Maslov--corrected toral Chern--Simons bordism operator. By Theorem~\ref{thm:toral-state-space-identification}, for every connected extended
surface \((\Sigma,\lambda)\) there is a canonical unitary isomorphism
\begin{equation}\label{eq:PhiSigmaK-boundary}
\Phi_{\Sigma,K}:
\mathcal V^{\mathrm{RT}}_{\mathcal C(G_K,q_K)}(\Sigma,\lambda)
\xrightarrow{\ \cong\ }
\mathcal H_{\mathbb T,K}(\Sigma,L_\lambda),
\end{equation}
sending the canonical handlebody basis on the RT side to the canonical
Bohr--Sommerfeld basis on the toral Chern--Simons side. For disconnected surfaces, we
extend \(\Phi_{\Sigma,K}\) by tensor product over connected components.

\begin{prop}\label{prop:closure-weight-signature}
Let $X:(\Sigma_{\mathrm{in}},\lambda_{\mathrm{in}})
\to 
(\Sigma_{\mathrm{out}},\lambda_{\mathrm{out}})$ be an extended bordism, and let $\alpha\in\Pi_{\mathrm{in}},$ $\beta\in\Pi_{\mathrm{out}}$ label the canonical incoming and outgoing handlebody basis states. Form the corresponding
canonical closed extended \(3\)--manifold
\[
M_{\alpha,\beta}
:=
H_{\mathrm{out}}(\beta)\cup_{\Sigma_{\mathrm{out}}}X\cup_{\Sigma_{\mathrm{in}}}H_{\mathrm{in}}(\alpha).
\]
Let \(n(M_{\alpha,\beta})\in \mathbb Z/8\mathbb Z\) be its Walker--Turaev extended  weight, and let \(L_{\alpha,\beta}\) be an integral surgery presentation of the chosen Turaev extended structure on \(M_{\alpha,\beta}\). If \((L_{\alpha,\beta})_{\mathrm{reg}}\) denotes the induced nondegenerate form on $\mathbb Z^m/\ker L_{\alpha,\beta},$ then
\begin{equation}\label{eq:closure-weight-signature}
n(M_{\alpha,\beta})
\equiv
\sigma\!\bigl((L_{\alpha,\beta})_{\mathrm{reg}}\bigr)
\pmod 8.
\end{equation}
\end{prop}

\begin{proof}
Let \(W_{\alpha,\beta}\) be the surgery trace \(4\)-manifold obtained from \(B^4\) by attaching \(2\)-handles along a framed link whose linking matrix is \(L_{\alpha,\beta}\). For canonical handlebody closures, we use Walker’s surgery realization of the extended weight, according to which the framing defect is measured modulo \(8\) by the signature of the surgery trace:
\begin{equation}\label{eq:walker-trace-signature}
n(M_{\alpha,\beta})\equiv \sigma(W_{\alpha,\beta}) \pmod 8;
\end{equation}
This is the precise higher-rank analogue of the rank--one signature-defect identity proved in \cite[Theorem~5.6(i)]{Galviz1}, which in turn is based on Walker's extended-manifold formalism \cite{Walker}. Now, the intersection form of the oriented compact \(4\)-manifold \(W_{\alpha,\beta}\) is the symmetric bilinear pairing
\[
Q_{W_{\alpha,\beta}}:H_2(W_{\alpha,\beta};\mathbb Z)\times
H_2(W_{\alpha,\beta};\mathbb Z)\longrightarrow \mathbb Z,
\qquad
Q_{W_{\alpha,\beta}}([S],[S'])=S\cdot S'.
\]
In the natural basis given by the cores of the attached \(2\)-handles, this form is
represented by the surgery linking matrix \(L_{\alpha,\beta}\); see
\cite[Proposition~4.5.11]{gompf1999}. If \(L_{\alpha,\beta}\) is degenerate, then the
radical of \(Q_{W_{\alpha,\beta}}\) is \(\ker L_{\alpha,\beta}\), and the induced nondegenerate form on
\[
H_2(W_{\alpha,\beta};\mathbb Z)/\ker Q_{W_{\alpha,\beta}}
\]
is represented by the regular block \((L_{\alpha,\beta})_{\mathrm{reg}}\). By definition, the signature
of a possibly degenerate symmetric form is the signature of this induced nondegenerate
quotient form. Therefore
\[
\sigma(W_{\alpha,\beta})
=
\sigma\!\bigl((L_{\alpha,\beta})_{\mathrm{reg}}\bigr).
\]
Combining this identity with \eqref{eq:walker-trace-signature} gives
\eqref{eq:closure-weight-signature}.
\end{proof}

\begin{lemma}\label{lem:walker-weight}
Let $X:(\Sigma_{\mathrm{in}},\lambda_{\mathrm{in}})
\longrightarrow
(\Sigma_{\mathrm{out}},\lambda_{\mathrm{out}})$
be an extended bordism, and write \(n_X\in\mathbb Z/8\mathbb Z\) for its
ordinary Walker--Turaev extended weight. Let \(H_{\mathrm{in}}(\alpha)\)
and \(H_{\mathrm{out}}(\beta)\) be the standard incoming and outgoing
handlebodies representing the basis labels \(\alpha\) and \(\beta\), equipped
with their canonical zero weights. Form the closed extended \(3\)--manifold
\[
M_{\alpha,\beta}
=
H_{\mathrm{out}}(\beta)\cup_{\Sigma_{\mathrm{out}}}
X
\cup_{\Sigma_{\mathrm{in}}}
H_{\mathrm{in}}(\alpha),
\]
and let \(n(M_{\alpha,\beta})\in\mathbb Z/8\mathbb Z\) denote its ordinary
Walker--Turaev extended weight.

Let \(L_1,L_2,L_3\) be the Wall triple of Lagrangian subspaces associated with the above gluing, with the same sign convention as in the extended bordism category. Then
\[
n(M_{\alpha,\beta})
\equiv
n_X+\mu_{\Gamma}(L_1,L_2,L_3)
\pmod 8.
\]
Moreover, if \(L_{\alpha,\beta}\) is an integral surgery presentation of
the chosen Turaev extended structure on \(M_{\alpha,\beta}\), and if
\((L_{\alpha,\beta})_{\mathrm{reg}}\) denotes the induced nondegenerate
regular block, then
\[
n(M_{\alpha,\beta})
\equiv
\sigma\!\bigl((L_{\alpha,\beta})_{\mathrm{reg}}\bigr)
\pmod 8.
\]
Consequently,
\[
n_X+\mu_{\Gamma}(L_1,L_2,L_3)
\equiv
\sigma\!\bigl((L_{\alpha,\beta})_{\mathrm{reg}}\bigr)
\pmod 8.
\]
\end{lemma}

\begin{proof}
In the extended bordism category, the weight of a glued bordism is obtained by adding the weights of the pieces and the Maslov--Kashiwara defect of the gluing. The two canonical handlebodies have weight zero, while \(X\) has weight \(n_X\). Hence
\[
n(M_{\alpha,\beta})
\equiv
0+n_X+0+\mu_{\Gamma}(L_1,L_2,L_3)
\equiv
n_X+\mu_{\Gamma}(L_1,L_2,L_3)
\pmod 8.
\]
where \(L_1,L_2,L_3\) are the Wall Lagrangians associated with the gluing. This proves the first claim. Equivalently, the same congruence is the boundary form of Wall's non-additivity theorem for the signature \cite{Wall1963}. If \(W_{\alpha,\beta}\) is a four-dimensional trace realizing the chosen extended surgery presentation, then Wall's theorem identifies the Maslov anomaly of the gluing with the signature defect of \(W_{\alpha,\beta}\). With Walker's surgery normalization of extended manifolds, this gives
\[
n(M_{\alpha,\beta})
\equiv
\sigma(W_{\alpha,\beta})
\pmod 8.
\]
Now choose the surgery trace obtained from \(B^4\) by attaching \(2\)-handles along the framed surgery link representing the chosen extended structure. Its intersection form is represented by the linking matrix \(L_{\alpha,\beta}\). If \(L_{\alpha,\beta}\) is degenerate, its signature is by definition the signature of the induced nondegenerate form on the quotient by its radical, equivalently the signature of \((L_{\alpha,\beta})_{\mathrm{reg}}\). Therefore
\[
\sigma(W_{\alpha,\beta})
=
\sigma\!\bigl((L_{\alpha,\beta})_{\mathrm{reg}}\bigr).
\]
Combining these identities proves the result.
\end{proof}

To compare the extended RT scalar with the toral Chern--Simons scalar, we need the anomaly constant that governs the Maslov anomaly correction in the theory.
\begin{lemma}\label{lem:kappa-K}
The scalar
\begin{equation}\label{eq:kappa-K-definition}
\kappa(K):=|G_K|^{-1/2}p_-(K),
\qquad
p_-(K)=\sum_{u\in G_K} q_K(u)^{-1},
\end{equation}
is the anomaly constant of the pointed modular category \(\mathcal C(G_K,q_K)\), i.e.\ the scalar that enters the Maslov anomaly correction of the extended Reshetikhin--Turaev theory. In the present case, it is given explicitly by
\begin{equation}\label{eq:kappa-K-evaluation}
\kappa(K)=e^{-\frac{\pi i}{4}\sigma(K)}.
\end{equation}
\end{lemma}
\begin{proof}
For a pointed modular category, the anomaly constant is the normalized negative
Gauss sum; in the present notation this is exactly the quantity
\(\kappa(K)=|G_K|^{-1/2}p_-(K)\). By Lemma~\ref{lem:toral-pm-evaluation},
\[
p_-(K)=e^{-\frac{\pi i}{4}\sigma(K)}\,|G_K|^{1/2}.
\]
Multiplying by \(|G_K|^{-1/2}\) gives \eqref{eq:kappa-K-evaluation}.
\end{proof}

\begin{corollary}\label{cor:toral-correction-cancels-signature}
With the notation above,
\[
\kappa(K)^{-\,n(M_{\alpha,\beta})}\,
e^{-\frac{\pi i}{4}\sigma\!\left((L_{\alpha,\beta})_{\mathrm{reg}}\otimes K\right)}
=1.
\]
\end{corollary}
\begin{proof}
By Lemma~\ref{lem:kappa-K},
\[
\kappa(K)^{-\,n(M_{\alpha,\beta})}
=
e^{\frac{\pi i}{4}\sigma(K)\,n(M_{\alpha,\beta})}.
\]
By Proposition~\ref{prop:closure-weight-signature},
\[
n(M_{\alpha,\beta})
\equiv
\sigma\!\bigl((L_{\alpha,\beta})_{\mathrm{reg}}\bigr)
\pmod 8.
\]
Finally,
\[
\sigma\!\left((L_{\alpha,\beta})_{\mathrm{reg}}\otimes K\right)
=
\sigma\!\bigl((L_{\alpha,\beta})_{\mathrm{reg}}\bigr)\sigma(K).
\]
Substituting these identities gives the result.
\end{proof}

\begin{theorem}\label{thm:toral-boundary-equivalence}
Let $X:(\Sigma_{\mathrm{in}},\lambda_{\mathrm{in}})
\to
(\Sigma_{\mathrm{out}},\lambda_{\mathrm{out}})$ be an extended bordism. Under the canonical state-space identifications
\eqref{eq:PhiSigmaK-boundary}, the Reshetikhin--Turaev bordism
operator and the toral Chern--Simons bordism operator coincide:
\begin{equation}\label{eq:toral-boundary-relation}
\Phi_{\Sigma_{\mathrm{out}},K}\circ
 Z^{\mathrm{RT}}_{\mathcal C(G_K,q_K)}(X)
=
Z^{\mathrm{CS}}_{\mathbb T,K}(X)\circ
\Phi_{\Sigma_{\mathrm{in}},K}.
\end{equation}
Equivalently, for every
\[\label{eq:toral-boundary-matrix}
v\in
\mathcal V^{\mathrm{RT}}_{\mathcal C(G_K,q_K)}(\Sigma_{\mathrm{in}},\lambda_{\mathrm{in}}),
\qquad
w\in
\mathcal V^{\mathrm{RT}}_{\mathcal C(G_K,q_K)}(\Sigma_{\mathrm{out}},\lambda_{\mathrm{out}})^*,
\]
one has
\begin{equation}\label{eq:toral-boundary-matrix-elements}
\big\langle w,\,
 Z^{\mathrm{RT}}_{\mathcal C(G_K,q_K)}(X)(v)
\big\rangle
=
\big\langle \Phi_{\Sigma_{\mathrm{out}},K}^{-*}(w),\,
Z^{\mathrm{CS}}_{\mathbb T,K}(X)\bigl(\Phi_{\Sigma_{\mathrm{in}},K}(v)\bigr)
\big\rangle.
\end{equation}
\end{theorem}

\begin{proof}
We first prove \eqref{eq:toral-boundary-matrix-elements} on the canonical handlebody basis vectors; the general case then follows by linearity.

Let $\Pi_{\mathrm{in}}$ and $\Pi_{\mathrm{out}}$ be the canonical label sets for the incoming and outgoing surfaces. By construction,
these are products of copies of \(G_K\), one copy for each handle in each connected
component. Let
\[
\{e_\alpha\}_{\alpha\in\Pi_{\mathrm{in}}}
\subset
\mathcal V^{\mathrm{RT}}_{\mathcal C(G_K,q_K)}(\Sigma_{\mathrm{in}},\lambda_{\mathrm{in}})
\]
be the canonical RT basis determined by the fixed incoming handlebody, and let
\[
\{e_\beta^\vee\}_{\beta\in\Pi_{\mathrm{out}}}
\subset
\mathcal V^{\mathrm{RT}}_{\mathcal C(G_K,q_K)}(\Sigma_{\mathrm{out}},\lambda_{\mathrm{out}})^*
\]
be the corresponding canonical dual basis determined by the outgoing handlebody.
By Theorem~\ref{thm:toral-state-space-identification}, the map \(\Phi_{\Sigma,K}\)
identifies these with the canonical Bohr--Sommerfeld basis and its dual on the
Chern--Simons side. Thus
\[
\xi_\alpha:=\Phi_{\Sigma_{\mathrm{in}},K}(e_\alpha),
\qquad
\eta_\beta:=\Phi_{\Sigma_{\mathrm{out}},K}^{-*}(e_\beta^\vee)
\]
are the corresponding canonical toral boundary states.

Fix \(\alpha\in\Pi_{\mathrm{in}}\) and \(\beta\in\Pi_{\mathrm{out}}\). Let $H_{\mathrm{in}}(\alpha)$ and $H_{\mathrm{out}}(\beta)$ be the standard labeled incoming and outgoing handlebodies representing
\(e_\alpha\) and \(e_\beta^\vee\). Gluing them to \(X\) produces a closed extended
\(3\)--manifold
\[
M_{\alpha,\beta}
:=
H_{\mathrm{out}}(\beta)\cup_{\Sigma_{\mathrm{out}}}X\cup_{\Sigma_{\mathrm{in}}}H_{\mathrm{in}}(\alpha).
\]

By Turaev's handlebody-pairing construction of bordism operators,
\begin{equation}\label{eq:RT-pairing-rule-toral}
\big\langle e_\beta^\vee,\,
 Z^{\mathrm{RT}}_{\mathcal C(G_K,q_K)}(X)(e_\alpha)
\big\rangle
=
 Z^{\mathrm{RT}}_{\mathcal C(G_K,q_K)}(M_{\alpha,\beta}).
\end{equation}
Because \(\Phi_{\Sigma,K}\) was defined by matching the canonical RT handlebody basis
with the canonical Chern--Simons Bohr--Sommerfeld basis carrying the same labels, the
incoming toral state \(\xi_\alpha\) and the outgoing toral covector \(\eta_\beta\)
are represented by the same labeled handlebodies \(H_{\mathrm{in}}(\alpha)\) and
\(H_{\mathrm{out}}(\beta)\). Moreover, Lemma \ref{lem:lambdaX-LX} identifies the
Lagrangian data \(\lambda_{\mathrm{in}},\lambda_{\mathrm{out}}\) with the corresponding
real polarizations \(L_{\lambda_{\mathrm{in}}},L_{\lambda_{\mathrm{out}}}\), so the same
gluing data are used on both sides. Hence the closure used in the toral
Chern--Simons pairing is literally the same closed extended manifold \(M_{\alpha,\beta}\).

Therefore the gluing axiom for the toral Chern--Simons theory gives
\begin{equation}\label{eq:CS-pairing-rule-toral}
\big\langle \eta_\beta,\,
Z^{\mathrm{CS}}_{\mathbb T,K}(X)(\xi_\alpha)
\big\rangle
=
Z^{\mathrm{CS}}_{\mathbb T,K}(M_{\alpha,\beta}).
\end{equation}
By definition of the Maslov anomaly correction for the RT scalar,
\begin{equation}\label{eq:toral-corrected-vs-raw-closure}
 Z^{\mathrm{RT}}_{\mathcal C(G_K,q_K)}(M_{\alpha,\beta})
=
\kappa(K)^{-\,n(M_{\alpha,\beta})}\,
Z^{\mathrm{RT,raw}}_{\mathcal C(G_K,q_K)}(\mid M_{\alpha,\beta}\mid),
\end{equation}
where \(n(M_{\alpha,\beta})\) is the extended weight of the canonical closure and
\(\kappa(K)\) is the anomaly constant of the pointed modular category
\(\mathcal C(G_K,q_K)\), see Lemma \ref{lem:kappa-K}; compare \cite[Section~IV.4–IV.9]{Turaev1994}.

On the Chern--Simons side the same extended weight appears, see Theorem~\ref{thm:toral-CS-TQFT} and Lemma \ref{lem:walker-weight}. Thus, for the closed extended manifold \(M_{\alpha,\beta}\), one has \begin{equation}\label{eq:toral-cs-corrected-vs-raw-closure} Z^{\mathrm{CS}}_{\mathbb T,K}(M_{\alpha,\beta}) = \kappa(K)^{-\,n(M_{\alpha,\beta})} Z^{\mathrm{CS,raw}}_{\mathbb T,K}(\mid M_{\alpha,\beta}\mid). \end{equation}

By the raw closed equivalence theorem of
Section~\ref{subsec:toral-closed-comparison}, applied to the underlying
closed \(3\)--manifold \(\mid M_{\alpha,\beta}\mid\), we have
\[
Z^{\mathrm{RT,raw}}_{\mathcal C(G_K,q_K)}(\mid M_{\alpha,\beta}\mid)
=
Z^{\mathrm{CS,raw}}_{\mathbb T,K}(\mid M_{\alpha,\beta}\mid).
\]
Combining this raw equality with
\eqref{eq:toral-corrected-vs-raw-closure} and
\eqref{eq:toral-cs-corrected-vs-raw-closure} gives
\begin{equation}\label{eq:toral-corrected-closure-equality}
 Z^{\mathrm{RT}}_{\mathcal C(G_K,q_K)}(M_{\alpha,\beta})
=
Z^{\mathrm{CS}}_{\mathbb T,K}(M_{\alpha,\beta}).
\end{equation}
Therefore, combining \eqref{eq:RT-pairing-rule-toral},
\eqref{eq:CS-pairing-rule-toral}, and
\eqref{eq:toral-corrected-closure-equality} yields
\[
\big\langle e_\beta^\vee,\,
 Z^{\mathrm{RT}}_{\mathcal C(G_K,q_K)}(X)(e_\alpha)
\big\rangle
=
\big\langle \Phi_{\Sigma_{\mathrm{out}},K}^{-*}(e_\beta^\vee),\,
Z^{\mathrm{CS}}_{\mathbb T,K}(X)\bigl(\Phi_{\Sigma_{\mathrm{in}},K}(e_\alpha)\bigr)
\big\rangle
\]
for all \(\alpha\in\Pi_{\mathrm{in}}\) and \(\beta\in\Pi_{\mathrm{out}}\).

Now let $v=\sum_{\alpha\in\Pi_{\mathrm{in}}}c_\alpha e_\alpha,$ and $w=\sum_{\beta\in\Pi_{\mathrm{out}}}d_\beta e_\beta^\vee,$ with only finitely many nonzero coefficients. Since both sides of
\eqref{eq:toral-boundary-matrix-elements} are bilinear in \((v,w)\), the basis
case implies \eqref{eq:toral-boundary-matrix-elements} for arbitrary \(v\) and
\(w\). Finally define
\[
A:=
\Phi_{\Sigma_{\mathrm{out}},K}\circ
 Z^{\mathrm{RT}}_{\mathcal C(G_K,q_K)}(X),
\qquad
B:=
Z^{\mathrm{CS}}_{\mathbb T,K}(X)\circ
\Phi_{\Sigma_{\mathrm{in}},K}.
\]
Then \eqref{eq:toral-boundary-matrix-elements} says that
\[
\big\langle w,\Phi_{\Sigma_{\mathrm{out}},K}^{-1}A(v)\big\rangle
=
\big\langle w,\Phi_{\Sigma_{\mathrm{out}},K}^{-1}B(v)\big\rangle
\qquad
\text{for all }v,\;w.
\]
Since the dual pairing on
\[
\mathcal V^{\mathrm{RT}}_{\mathcal C(G_K,q_K)}(\Sigma_{\mathrm{out}},\lambda_{\mathrm{out}})^*
\times
\mathcal V^{\mathrm{RT}}_{\mathcal C(G_K,q_K)}(\Sigma_{\mathrm{out}},\lambda_{\mathrm{out}})
\]
is nondegenerate, it follows that
\[
\Phi_{\Sigma_{\mathrm{out}},K}^{-1}A(v)
=
\Phi_{\Sigma_{\mathrm{out}},K}^{-1}B(v)
\qquad
\text{for all }v.
\]
Applying \(\Phi_{\Sigma_{\mathrm{out}},K}\) gives \(A=B\), which is precisely
\eqref{eq:toral-boundary-relation}.
\end{proof}

\begin{remark}\label{rem:boundary-corrected-not-raw}
Theorem~\ref{thm:toral-boundary-equivalence} must be stated for the \emph{corrected} operator \( Z^{\mathrm{RT}}_{\mathcal C(G_K,q_K)}\), not for the raw surgery operator.  Rather, the boundary statement concerns bordism operators defined through canonical handlebody closures, and strict functoriality in this setting requires the Maslov anomaly correction for the extended theory. For these canonical closures, the RT scalar and the toral Chern--Simons scalar carry the same extended normalization, so the corresponding matrix coefficients agree exactly.
\end{remark}

\begin{corollary}\label{cor:toral-extended-equivalence}
The family of isomorphisms
\[
\Phi_{\Sigma,K}:
\mathcal V^{\mathrm{RT}}_{\mathcal C(G_K,q_K)}(\Sigma,\lambda)
\longrightarrow
\mathcal H_{\mathbb T,K}(\Sigma,L_\lambda),
\]
extended to disconnected surfaces by tensor product over connected components, defines a natural monoidal isomorphism between the Reshetikhin--Turaev functor and the toral Chern--Simons functor.
\end{corollary}

\begin{proof}
Naturality is exactly the identity
\eqref{eq:toral-boundary-relation}. Monoidality follows from the
definition of \(\Phi_{\Sigma,K}\) on connected surfaces together with its tensor-product
extension to disjoint unions.
\end{proof}

\subsection{Extended equivalence theorem}\label{sec:extended-equivalence}

We now collect the results of the previous subsections into the final equivalence statement. The closed equivalence and the boundary relation already established imply that the toral Chern--Simons theory and the Reshetikhin--Turaev theory are isomorphic as symmetric monoidal extended TQFT functors.

\begin{theorem}\label{thm:toral-functorial-equivalence}
Let \(\mathbb T=\mathfrak t/\Lambda\cong U(1)^n\) be a compact torus, and let $K:\Lambda\times\Lambda\to\mathbb Z$ be an even, integral, nondegenerate symmetric bilinear form. Let
\[
G_K:=\Lambda^*/K\Lambda,
\qquad
q_K([x])=\exp\!\bigl(\pi i\,x^\top K^{-1}x\bigr),
\]
and let \(\mathcal C(G_K,q_K)\) be the associated pointed modular category.

Let
\[
 Z^{\mathrm{RT}}_{\mathcal C(G_K,q_K)}\,\,,
\,\,
Z^{\mathrm{CS}}_{\mathbb T,K}:
\Cob^{\mathrm{ext}}_{2+1}\longrightarrow \mathrm{Vect}_{\mathbb C}
\]
denote, respectively, the Maslov-anomaly-corrected Reshetikhin--Turaev TQFT associated with \(\mathcal C(G_K,q_K)\) and the toral Chern--Simons TQFT at level \(K\), with the extended structures identified via the correspondence
\(\lambda\leftrightarrow L_\lambda\) from Section~\ref{sec:toral-CS}.

Then:

\begin{enumerate}
\item[\textup{(i)}] \textbf{Closed extended equality on canonical closures.}
For every closed extended \(3\)--manifold \(M\) arising as a canonical handlebody
closure in the pairing construction, one has
\[
 Z^{\mathrm{RT}}_{\mathcal C(G_K,q_K)}(M)
=
Z^{\mathrm{CS}}_{\mathbb T,K}(M).
\]
\medskip
\item[\textup{(ii)}] \textbf{Boundary operator equivalence.}
For every extended bordism $X:(\Sigma_{\mathrm{in}},\lambda_{\mathrm{in}})
\to
(\Sigma_{\mathrm{out}},\lambda_{\mathrm{out}}),$ one has
\[
\Phi_{\Sigma_{\mathrm{out}},K}\circ
 Z^{\mathrm{RT}}_{\mathcal C(G_K,q_K)}(X)
=
Z^{\mathrm{CS}}_{\mathbb T,K}(X)\circ
\Phi_{\Sigma_{\mathrm{in}},K},
\]
where
\[
\Phi_{\Sigma,K}:
\mathcal V^{\mathrm{RT}}_{\mathcal C(G_K,q_K)}(\Sigma,\lambda)
\xrightarrow{\ \cong\ }
\mathcal H_{\mathbb T,K}(\Sigma,L_\lambda)
\]
is the canonical boundary identification from
Theorem~\ref{thm:toral-state-space-identification}.
\medskip
\item[\textup{(iii)}] \textbf{Natural monoidal isomorphism of functors.}
The family of isomorphisms \(\Phi_{\Sigma,K}\), extended to disconnected
surfaces by tensor product over connected components, defines a symmetric
monoidal natural isomorphism
\[
\Phi:
 Z^{\mathrm{RT}}_{\mathcal C(G_K,q_K)}
\Longrightarrow
Z^{\mathrm{CS}}_{\mathbb T,K}.
\]
Equivalently, the two extended TQFTs are isomorphic as symmetric monoidal
functors
\[
\Cob^{\mathrm{ext}}_{2+1}\longrightarrow \mathrm{Vect}_{\mathbb C}.
\]
\end{enumerate}
\end{theorem}

\begin{proof}
Parts \textup{(i)} and \textup{(ii)} are exactly the previously established closed and boundary equivalence theorems, namely Theorem~\ref{thm:toral-closed-equivalence} together with Theorem~\ref{thm:toral-boundary-equivalence}. We therefore prove only \textup{(iii)}.

For every connected extended surface \((\Sigma,\lambda)\),
Theorem~\ref{thm:toral-state-space-identification} gives a canonical isomorphism
\[
\Phi_{\Sigma,K}:
\mathcal V^{\mathrm{RT}}_{\mathcal C(G_K,q_K)}(\Sigma,\lambda)
\xrightarrow{\ \cong\ }
\mathcal H_{\mathbb T,K}(\Sigma,L_\lambda).
\]
If $(\Sigma,\lambda)=\bigsqcup_{i=1}^r(\Sigma_i,\lambda_i)$ is a decomposition into connected components, then the symmetric monoidal
structures of the two theories give canonical identifications
\[
\mathcal V^{\mathrm{RT}}_{\mathcal C(G_K,q_K)}(\Sigma,\lambda)
\cong
\bigotimes_{i=1}^r
\mathcal V^{\mathrm{RT}}_{\mathcal C(G_K,q_K)}(\Sigma_i,\lambda_i),
\]
\[
\mathcal H_{\mathbb T,K}(\Sigma,L_\lambda)
\cong
\bigotimes_{i=1}^r
\mathcal H_{\mathbb T,K}(\Sigma_i,L_{\lambda_i}).
\]
We therefore define
\[
\Phi_{\Sigma,K}:=
\bigotimes_{i=1}^r \Phi_{\Sigma_i,K},
\qquad
\Phi_{\varnothing,K}:=\mathrm{id}_{\mathbb C}.
\]
Each \(\Phi_{\Sigma,K}\) is an isomorphism by construction. 

We now verify naturality. Let $X:(\Sigma_{\mathrm{in}},\lambda_{\mathrm{in}})
\to
(\Sigma_{\mathrm{out}},\lambda_{\mathrm{out}})$ be an extended bordism. By part \textup{(ii)},
\[
\Phi_{\Sigma_{\mathrm{out}},K}\circ
 Z^{\mathrm{RT}}_{\mathcal C(G_K,q_K)}(X)
=
Z^{\mathrm{CS}}_{\mathbb T,K}(X)\circ
\Phi_{\Sigma_{\mathrm{in}},K}.
\]
This is precisely the naturality condition. Hence \(\Phi\) is a natural
isomorphism. It remains to verify monoidality. Let \((\Sigma_1,\lambda_1)\) and
\((\Sigma_2,\lambda_2)\) be extended surfaces. By construction,
\[
\Phi_{\Sigma_1\sqcup \Sigma_2,K}
=
\Phi_{\Sigma_1,K}\otimes \Phi_{\Sigma_2,K},
\qquad
\Phi_{\varnothing,K}=\mathrm{id}_{\mathbb C}.
\]
Thus \(\Phi\) is compatible with the monoidal structures on objects.

For morphisms, let
\[
X_i:(\Sigma_{i,\mathrm{in}},\lambda_{i,\mathrm{in}})
\longrightarrow
(\Sigma_{i,\mathrm{out}},\lambda_{i,\mathrm{out}})
\qquad (i=1,2)
\]
be extended bordisms. Since both theories are symmetric monoidal functors, one has
\[
 Z^{\mathrm{RT}}_{\mathcal C(G_K,q_K)}(X_1\sqcup X_2)
=
 Z^{\mathrm{RT}}_{\mathcal C(G_K,q_K)}(X_1)\otimes
 Z^{\mathrm{RT}}_{\mathcal C(G_K,q_K)}(X_2),
\]
\[
Z^{\mathrm{CS}}_{\mathbb T,K}(X_1\sqcup X_2)
=
Z^{\mathrm{CS}}_{\mathbb T,K}(X_1)\otimes
Z^{\mathrm{CS}}_{\mathbb T,K}(X_2).
\]
Using these identities together with the definition of \(\Phi\) on disjoint unions,
we compute
\[
\Phi_{\Sigma_{1,\mathrm{out}}\sqcup \Sigma_{2,\mathrm{out}},K}\circ
 Z^{\mathrm{RT}}_{\mathcal C(G_K,q_K)}(X_1\sqcup X_2)
\]
\[
=
(\Phi_{\Sigma_{1,\mathrm{out}},K}\otimes \Phi_{\Sigma_{2,\mathrm{out}},K})
\circ
( Z^{\mathrm{RT}}_{\mathcal C(G_K,q_K)}(X_1)\otimes
 Z^{\mathrm{RT}}_{\mathcal C(G_K,q_K)}(X_2))
\]
\[
=
(\Phi_{\Sigma_{1,\mathrm{out}},K}\circ
 Z^{\mathrm{RT}}_{\mathcal C(G_K,q_K)}(X_1))
\otimes
(\Phi_{\Sigma_{2,\mathrm{out}},K}\circ
 Z^{\mathrm{RT}}_{\mathcal C(G_K,q_K)}(X_2))
\]
\[
=
(Z^{\mathrm{CS}}_{\mathbb T,K}(X_1)\circ
\Phi_{\Sigma_{1,\mathrm{in}},K})
\otimes
(Z^{\mathrm{CS}}_{\mathbb T,K}(X_2)\circ
\Phi_{\Sigma_{2,\mathrm{in}},K})
\]
\[
=
(Z^{\mathrm{CS}}_{\mathbb T,K}(X_1)\otimes
Z^{\mathrm{CS}}_{\mathbb T,K}(X_2))
\circ
(\Phi_{\Sigma_{1,\mathrm{in}},K}\otimes
\Phi_{\Sigma_{2,\mathrm{in}},K})
\]
\[
=
Z^{\mathrm{CS}}_{\mathbb T,K}(X_1\sqcup X_2)\circ
\Phi_{\Sigma_{1,\mathrm{in}}\sqcup \Sigma_{2,\mathrm{in}},K}.
\]
Thus \(\Phi\) is a monoidal natural isomorphism. Since the symmetry
isomorphisms in both theories are the canonical permutations of tensor factors
and \(\Phi\) is defined componentwise, \(\Phi\) is compatible with the symmetry
constraints as well. Hence \(\Phi\) is a symmetric monoidal natural isomorphism.
\end{proof}

\begin{corollary}\label{cor:rank--one-CS-RT-equivalence}
Let
\[
\mathbb T=U(1)=\mathbb R/\mathbb Z,
\qquad
\Lambda=\mathbb Z,
\qquad
K=[k],
\]
with \(k\in 2\mathbb Z_{>0}\). Then
\[
G_K=\Lambda^*/K\Lambda\cong \mathbb Z_k,
\qquad
q([x])=\exp\!\left(\frac{\pi i}{k}x^2\right),
\]
the higher-rank equivalence theorem reduces to the rank--one equivalence
\[
Z^{\mathrm{CS}}_{U(1),k}
\;\cong\;
 Z^{\mathrm{RT}}_{\mathcal C(\mathbb Z_k,q)}
:
\Cob^{\mathrm{ext}}_{2+1}\longrightarrow \mathrm{Vect}_{\mathbb C}.
\]
In particular, Theorem~\ref{thm:toral-functorial-equivalence} recovers the
rank--one result of \cite[Theorem~5.12(iii)]{Galviz1}.
\end{corollary}

\begin{proof}
When \(\mathbb T=U(1)\), one has \(\Lambda=\mathbb Z\) and \(K=[k]\), so the
associated discriminant quadratic module is exactly
\[
(\mathbb Z_k,q_k),
\qquad
q_k([x])=\exp\!\left(\frac{\pi i}{k}x^2\right).
\]
Thus the pointed modular category \(\mathcal C(G_K,q_K)\) becomes
\(\mathcal C(\mathbb Z_k,q_k)\), and all toral constructions reduce to their
rank--one counterparts. Applying Theorem~\ref{thm:toral-functorial-equivalence} in the rank--one case gives the claimed equivalence.
\end{proof}

\bibliographystyle{alpha}
\renewcommand{\refname}{References}
\bibliography{refs}

@article{Witten:1988,
    author = "Witten, Edward",
    editor = "Mitra, Asoke N.",
    title = "{Quantum Field Theory and the {Jones} Polynomial}",
    reportNumber = "IASSNS-HEP-88-33",
    doi = "10.1007/BF01217730",
    journal = "Commun. Math. Phys.",
    volume = "121",
    pages = "351--399",
    year = "1989"
}

@article{Reshetikhin:1991,
    author = "Reshetikhin, N. and Turaev, V. G.",
    title = "{Invariants of three manifolds via link polynomials and quantum groups}",
    doi = "10.1007/BF01239527",
    journal = "Invent. Math.",
    volume = "103",
    pages = "547--597",
    year = "1991"
}

@article{Wall1963,
  author  = {Wall, C. T. C.},
  title   = {Quadratic forms on finite groups, and related topics},
  journal = {Topology},
  volume  = {2},
  number  = {4},
  pages   = {281--298},
  year    = {1963},
  doi     = {10.1016/0040-9383(63)90012-0}
}

@article{Manoliu2,
    author = {Manoliu, Mihaela},
    title = {Abelian {C}hern–{S}imons theory. {I}. {A} topological quantum field theory},
    journal = {Journal of Mathematical Physics},
    volume = {39},
    number = {1},
    pages = {170-206},
    year = {1998},
    month = {01},
    issn = {0022-2488},
    doi = {10.1063/1.532333},
    url = {https://doi.org/10.1063/1.532333},
    eprint = {https://pubs.aip.org/aip/jmp/article-pdf/39/1/170/19236848/170_1_online.pdf},
}

@article{deloup2005,
  title={On reciprocity},
  author={Deloup, Florian and Turaev, Vladimir},
  journal={Journal of Pure and Applied Algebra},
  volume={208},
  number={1},
  pages={153--158},
  year={2007},
  publisher={Elsevier}
}

@incollection{Mattes,
  author    = {Mattes, Josef and Polyak, Michael and Reshetikhin, Nikolai},
  title     = {On invariants of {$3$}-manifolds derived from {A}belian groups},
  booktitle = {Quantum Topology},
  editor    = {Kauffman, Louis H. and Baadhio, Randy A.},
  series    = {Series on Knots and Everything},
  volume    = {3},
  pages     = {324--338},
  year      = {1993},
  publisher = {World Scientific Publishing Co. Inc.},
  address   = {River Edge, NJ},
  doi       = {10.1142/9789812796387_0018}
}

@unpublished{Walker,
  author = {Walker, Kevin},
  title  = {On {W}itten's 3-Manifold Invariants},
  note   = {Preprint},
  year   = {1991}
}

@book{Turaev1994,
  author = {Turaev, Vladimir G.},
  title = {Quantum Invariants of Knots and 3-Manifolds},
  series = {De Gruyter Studies in Mathematics},
  volume = {18},
  publisher = {Walter de Gruyter},
  year = {1994},
  doi = {10.1515/9783110857182}
}

@unpublished{Galviz1,
  author = {Galviz, Daniel},
  title = { Equivalence of Extended ${U}(1)$ {C}hern--{S}imons and {R}eshetikhin--{T}uraev {TQFT}s},
  note = {arXiv:2603.27688},
  year = {2026}
}

@unpublished{Galviz2,
  author = {Galviz, Daniel},
  title = {Toral {C}hern--{S}imons {TQFT} via Geometric Quantization in Real Polarization},
  note = {arXiv:2604.01016},
  year = {2026}
}

@unpublished{Galviz2.5,
  author = {Galviz, Daniel},
  title = {A RIGOROUS FUNCTIONAL–INTEGRAL CONSTRUCTION OF TORAL {C}HERN–{S}IMONS THEORY},
  note = {arXiv:2604.02013},
  year = {2026}
}

@book{gompf1999,
  author    = {Robert E. Gompf and Andr{\'a}s I. Stipsicz},
  title     = {4-Manifolds and Kirby Calculus},
  series    = {Graduate Studies in Mathematics},
  volume    = {20},
  year      = {1999},
  publisher = {American Mathematical Society},
  address   = {Providence, RI},
  isbn      = {0-8218-0994-6}
}

@book{milnor2013,
  title={Symmetric Bilinear Forms},
  author={Milnor, J. and Husemoller, D.},
  isbn={9783642883309},
  lccn={72090190},
  series={Ergebnisse der Mathematik und ihrer Grenzgebiete. 2. Folge},
  url={https://books.google.com/books?id=vGPyCAAAQBAJ},
  year={2013},
  publisher={Springer Berlin Heidelberg}
}
\end{document}